\documentclass[]{amsart}

      \usepackage{amssymb,amsmath,latexsym,amsthm,}
      \usepackage[english]{babel}
      
      \usepackage[applemac]{inputenc}
      \usepackage{url}
      \usepackage[OT2,OT1]{fontenc}
      \newcommand\cyr{%
      \renewcommand\rmdefault{wncyr}%
      \renewcommand\sfdefault{wncyss}%
      \renewcommand\encodingdefault{OT2}%
      \normalfont
      \selectfont}
      \DeclareTextFontCommand{\textcyr}{\cyr}

      \theoremstyle{plain}
      \newtheorem{theorem}{Theorem}[section]
      \newtheorem{lemma}[theorem]{Lemma}
      \newtheorem{proposition}[theorem]{Proposition}
      \newtheorem{corollary}[theorem]{Corollary}
      
      \theoremstyle{definition}
      \newtheorem{definition}[theorem]{Definition}
      
      \theoremstyle{remark}
      \newtheorem{remark}[theorem]{Remark}

      \newcommand{\R}{{\mathbb R}}
      
      \newcommand{\C}{{\mathbb C}}

      \newcommand{\pv}{\mathop{\rm p.\,v. }\limits}
      
      \newfont{\cmbsy}{cmbsy10}
      \newfont{\cmmib}{cmmib10}
      \newcommand{\Orden}{\mathop{\hbox{\cmbsy O}}}

      \DeclareMathOperator{\li}{li}
      \DeclareMathOperator{\ali}{ali}

      \makeatletter
      \def\@setcopyright{}
      \def\serieslogo@{}
      \makeatother
   
\begin{document}

%


   
   \author{Juan Arias de Reyna$\,\,$}

   \address{Facultad de Matem\'aticas, 
   Universidad de Sevilla, \newline
   Apdo.~1160, 41080-Sevilla, Spain}
   \email{arias@us.es}


   \author{Jérémy Toulisse}
   \address{51 Route de Laborie,
   Foulayronnes 47510 (Lot-et-Garonne)
   France }
   \email{toulisse.jeremy@voila.fr}



   \title
   {The $n$-th prime asymptotically}


   \begin{abstract}
A new derivation of the classic asymptotic expansion of the 
$n$-th prime is presented. A fast algorithm for the computation of its 
terms is also given, which will be an improvement of that by Salvy (1994). 

Realistic bounds for the error with $\li^{-1}(n)$, after having retained  
the first $m$ terms, for $1\le m\le 11$, are given. 
Finally, assuming the Riemann Hypothesis, we give estimations 
of the best possible $r_3$ such that, for $n\ge r_3$, we have $p_n> s_3(n)$
where $s_3(n)$ is the sum of the first four terms of the asymptotic expansion.    
   \end{abstract}






   \maketitle

\section{Introduction.}
   
   \subsection{Historical note.}
   Chebyshev failed to fully prove the Prime Number Theorem (PNT), but he 
   obtained some notable approximations. For example, he proved 
   that for every natural number $n$: if the limit
   \begin{displaymath}
   \lim_{x\to\infty}\frac{\log^nx}{x}\bigl(\pi(x)-\li(x)\bigr)
   \end{displaymath}
   exists, then this limit must be equal to $0$. 
   
   The question was decided by  de la Vallée Poussin (1899) when he 
   gave his bound on the 
   error in the PNT: 
   The above limits exist and equal $0$.
   
   In 1894, Pervushin,
   a priest in Perm,  published several 
   formulae obtained empirically about prime numbers%
   \footnote{Ivan Mikheevich Pervushin (1827-1900)
   {\cyr  (Ivan Mikheevich Pervushin)}. No small achievement if 
   we note that he had only a table of primes up to $3\,000\,000$.}. 
   One of these formulae
   gives the following approximation to the $n$-th prime
   \begin{displaymath}
   \frac{p_n}{n}=\log n+\log\log n-1+\frac{5}{12\log n}+\frac{1}{24\log^2n}.
   \end{displaymath}
   Cesàro then published a note \cite[1894]{C} where he asserts that the 
   true formula is
   \begin{multline*}
   \frac{p_n}{n}=\log n+\log\log n-1+\frac{\log\log n-2}{\log n}-\\
   -\frac{(\log\log n)^2-6\log\log n+11}{2\log^2n}+o(\log^{-2}n).
   \end{multline*}
   Despite no mention by
   Cesàro in \cite{C}, the editors of his collected
   works added a note to \cite{C} pointing out that certain formulae quoted
   by Cesàro, since they followed from the results of Chebyshev, were only 
   established under the assumption of the existence 
   of the implied limits.  It therefore remains unsurprising that 
   Hilbert, in the Jahrbuch\footnote{\emph{Jahrbuch über die Fortschritte 
   der Mathematik (1868--1942)}, a forerunner for the 
   \emph{Zentralblatt für Mathematik}, at present digitalized  at 
   \url{http://www.emis.de/MATH/JFM/JFM.html}.} stated that Césaro did not prove his
   formula. 
   
   Landau \cite[1907]{L} several years later was better informed: 
   a formula, like that  
   of Cesàro, would imply the PNT, which had yet to be proved
   at Cesàro's time. However, using the results of Chebyshev, Cesàro may claim
   that if there is some formula for $p_n$ correct to the order 
   $n(\log n)^{-2}$, then it must coincide with his formula.
   
   Cipolla \cite[1902]{Ci} obtained an infinite asymptotic 
   expansion for $p_n$ and 
   gave a recursive formula to compute its terms.  He published after the 
   results of de la Vallée Poussin  but it seems that he was unaware of   
   these results, so that
   gave his proof under the same hypotheses as Cesáro.  So uninformed was he that 
   he attempted to prove some false formulae of Pervushin already \emph{corrected} by 
   Torelli \cite{T}
   \begin{displaymath}
   p_{n+1}-p_n=\log n+\log\log n+\frac{\log\log n-1}{\log n}+\Orden\Bigl(\frac
   {\log\log n}{\log n}\Bigr)^2
   \end{displaymath}
   with an \emph{impeccable proof} that if such a formula exists, then it must 
   be this formula.  (Such a formula would refute the twin prime conjecture,
   and today the above formula is known to be false.)
   
   In  his \emph{Handbuch}  Landau \cite[\S\ 57]{L2}  obtained by
   means of  the procedure of 
   Cesàro, some approximative formulae for $p_n$, and explained that the 
   method could give further terms. He also mentioned some recursive formulae
   without giving any clue for their derivation.
   
   We may say that Pervushin was the first to deal with  a formula for 
   $p_n$, albeit that he gave only the first few terms. 
   Cesàro then proved that in the
   case  such a formula exists, it must be one from which he would be able to derive
   several terms. Cipolla  found a method to write all the terms
   of the expansion if there is one.  Landau saw that the results of 
   de la Vallée Poussin imply that the expansion certainly exists.
         

   The algorithm given by Cipolla is not very convenient for the computation of
   the terms of the expansion. He iteratively computes the derivative
   of some polynomials appearing in the expansion but    computes 
   the constant terms as  determinants of increasing order. 
   Robin \cite[1988]{Ro2} considers the problem of computing these and other 
   similar expansions, leaving the problem of computing the 
   constant terms of the polynomials as an open problem.  Later 
   Salvy \cite[1994]{Sa} gives a satisfactory algorithm. This algorithm
   needs $\Orden(n^{7/2}\sqrt{\log n})$ coefficient operations to compute
   all the polynomials up to the $n$-th polynomial. 

   The asymptotic expansion of $p_n$ also plays a role in the study of $g(n)$,
   which is the maximum order of any element in the symmetric group $S_n$. 
   In fact, $\log g(n)$ has the same asymptotic expansion as $\sqrt{\li^{-1}(x)}$
   \cite{MNR}.

   There are many results giving true bounds on $p_n$, for example we  mention
   $p_n\ge n\log n$  \cite[1939]{R}, and $p_n\ge n(\log n+\log\log n-1)$ 
   \cite[1999]{D} both for $n\ge2$ (with partial results given in \cite{RS}, 
   \cite{Ro}, \cite{MR}, \cite{D}).
   In \cite{D2} it is also proved that
   \begin{displaymath}
   p_n
   \le n\Bigl(\log n+\log\log n-1+\frac{\log\log n-2}{\log n}\Bigr),\qquad n\ge 
   688\,383.
   \end{displaymath}

   \medskip
   
   \subsection{Organization of the paper.}
   In this paper we present a new derivation of the asymptotic
   expansion for $p_n$ and obtain explicit bounds for the error. 
   
   First, it must 
   be said that the asymptotic expansion has, in a certain sense, nothing
   to do with prime numbers: it is an asymptotic expansion of 
   $\ali(x):=\li^{-1}x$ which is the
   inverse of the usual logarithmic integral function.
   
   In Section \ref{S2} a proof of the existence of the expansion is given, 
   following  the path 
   of Cesàro, since it cannot be  found elsewhere, although
   it is frequently  claimed it can be done.   This Section is not needed
   in the rest of the paper.
   
   In Section \ref{S3},  a new formal derivation of the expansion is given.
   We obtain a new algorithm to compute the polynomials (Theorem \ref{T:algorithm}).  
   This is  simpler
   than that given by Salvy \cite{Sa}. Our algorithm  allows 
   all the polynomials 
   up to the $n$-th one to be computed in $\Orden(n^2)$ coefficient 
   operations (Theorem 
   \ref{T:complexity}). 
   It must be said that these polynomials have $\Orden(n^2)$ coefficients.

   In Section \ref{Sbounds}, independently of Section \ref{S2}, we prove that 
   the  formal expansion given in Section \ref{S3} is in fact
   the asymptotic expansion of $\ali(x)$ and gives  realistic bounds on
   the error (Theorem \ref{TMain} and \ref{concrete}).  
   
   In Section \ref{S4}, the results are applied to $p_n$ the $n$-th prime.
   Using the results of de la Vallée Poussin it can be shown that the asymptotic
   expansion of $\ali(n)$ is also an asymptotic expansion for $p_n$. 
   
   By assuming the Riemann Hypothesis, we found (Theorem \ref{Tdistance}) that 
   \begin{displaymath}
   |p_n-\ali(n)|\le\frac{1}{\pi}\sqrt{n}\,(\log  n)^{\frac52},\qquad n\ge11.
   \end{displaymath}
   This bound of $p_n$ is better than all the bounds cited above.
   
   We end the paper by motivating why  the above bounds have not
   been extended to further terms of the asymptotic expansion (Theorem \ref{T:r3}).
   \bigskip
   
   Notations:  With a certain hesitation we have introduced the notation
   $\ali(x)$ to denote the inverse function of $\li(x)$. 
   
   In Section \ref{Sbounds}, where  explicit bounds are sought, it has been 
   useful to denote by $\theta$ a real or complex number of absolute value
   $|\theta|\le1$, which will not always be the same, and depends on all parameters or 
   variables in the corresponding equation.
   
   \bigskip

   \textsc{Acknowledgement: } The authors would like to thank Jan van de Lune 
   ( Hallum, The Netherlands )
   for his linguistic assistance in
   preparing the paper, and his interest in our results.
   
\section{The inverse function of the logarithmic integral.}
   
   Usually $\li(x)$ is defined for real $x$ as the principal value 
   of the integral 
   \begin{displaymath}
   \li(x)=\pv\int_0^x\frac{dt}{\log t}.
   \end{displaymath}
   It may be extended  to an analytic function over the region $\Omega=
   \C\smallsetminus(-\infty,1]$, which is
   the complex plane with a cut along the real 
   axis $x\le 1$. The main branch of the logarithm is defined in $\Omega$ and does
   not vanish there. Therefore,  $\li(z)$ may be defined  in $\Omega$ by 
   \begin{equation}
   \li(z)=\li(2)+\int_2^z\frac{dt}{\log t},\qquad z\in\Omega
   \end{equation}
   where we integrate, for example, along the segment from $2$ to $z$. 
   
   For real $x>1$, the function $\li(x)$ is increasing and maps the 
   interval $(1,+\infty)$ onto $(-\infty,+\infty)$, so that we may define the inverse
   function $\ali\colon\R\to(1,+\infty)$ by
   \begin{equation}
   \li(\ali(x))=x.
   \end{equation}
   The function $\li(x)$ is analytic on $\Omega$, so that $\ali(x)$ is real 
   analytic. It is clear that we have the following rules of differentiation 
   \begin{equation}
   \frac{d}{dx}\li(x)=\frac{1}{\log x},\qquad \frac{d}{dx}\ali(x)=\log\ali(x).
   \end{equation}
   
   It is well known that the function $\li(x)$ has an asymptotic  expansion:
   \begin{theorem}
   For each integer $N\ge0$ 
   \begin{equation}\label{liexp}
   \li(x)=\frac{x}{\log x}\Bigl(1+\sum_{k=1}^N \frac{k!}{\log^k x}+
   \Orden\Bigl(\frac{1}{\log^{N+1}x}\Bigr)\Bigr),\qquad (x\to+\infty).
   \end{equation}
   \end{theorem}
   
   This may be proved  by repeated integration by parts
   (see \cite[p.~190--192]{MV}).    
   
   \section{Asymptotic expansion of $\ali(x)$.}\label{S2}
   
   In this section, we prove  the following
   \begin{theorem}\label{AE}
   For each integer $N\ge0$ 
   \begin{equation}\label{H}
   \frac{\ali(e^x)}{x e^x}=1+\sum_{n=1}^N\frac{P_{n-1}(\log x)}{x^n}+
   \Orden\Bigl(\frac{\log^{N+1}x}{x^{N+1}}\Bigr), \qquad (x\to+\infty)
   \end{equation}
   where the $P_{n-1}(z)$ are polynomials of degree $\le n$.
   \end{theorem}
   
   In the case of $N=0$ the sum must be understood as equal to $0$. 
   
   The theorem says that,
   for each $N$,  there exists an $x_N>1$ and a constant $C_N$ such that 
   \begin{displaymath}
   \Bigl|\frac{\ali(e^x)}{x e^x}-1-\sum_{n=1}^N\frac{P_{n-1}(\log x)}{x^n}
   \Bigr|\le C_N \frac{\log^{N+1}x}{x^{N+1}},\qquad (x>x_N).
   \end{displaymath}
   
   In the course of the proof we will make repeated use of the following
   \begin{lemma}\label{Lsubs}
   Let $f(x)$ be a function defined on a neighbourhood of $x=0$ such that
   \begin{equation}\label{L1}
   f(x)=a_1x+\cdots +a_N x^N+\Orden(x^{N+1}),\qquad (x\to0)
   \end{equation}
   where the $a_k$ are given constants. Assume that $g(x)$ satisfies
   \begin{equation}\label{L2}
   g(x)=\sum_{n=1}^N \frac{p_n(\log x)}{x^n}+
   \Orden\Bigl(\frac{\log^{N+1}x}{x^{N+1}}\Bigr),\qquad (x\to+\infty)
   \end{equation}
   where the $p_n(z)$ are polynomials of degree $\le n$.
   Then there exist polynomials $q_k(z)$ of degree $\le k$  such that
   \begin{equation}\label{L5}
   f(g(x))=\sum_{k=1}^N \frac{q_k(\log x)}{x^k}+
   \Orden\Bigl(\frac{\log^{N+1}x}{x^{N+1}}\Bigr),\qquad (x\to+\infty).
   \end{equation}
   \end{lemma}
   
   \begin{proof}
   It is clear that, for each $1\le n\le N$ we have $p_n(\log x)x^{-n}
   =\Orden((\log x/x)^n)$. Therefore, $g(x)=\Orden(\log x/x)$ (this is 
   true even when $N=0$ and there is no $p_n$). 
   It follows that $\lim_{x\to+\infty}g(x)=0$ and by substitution
   in \eqref{L1},  we obtain
   \begin{equation}\label{L3}
   f(g(x))=\sum_{n=1}^N a_n g(x)^n +\Orden\Bigl(\frac{\log ^{N+1}x}{x^{N+1}}
   \Bigr).
   \end{equation}
   By expanding the powers $g(x)^n$ by \eqref{L2} it is easy to  
   obtain an expression of the form   
   \begin{equation}
   g(x)^n=\sum_{k=1}^N \frac{p_{n,k}(\log x)}{x^k}+
   \Orden\Bigl(\frac{\log^{N+1}x}{x^{N+1}}\Bigr),\qquad (x\to+\infty)
   \end{equation}  
   where each $p_{n,k}(z)$ is a polynomial of degree $\le k$. 
   By substituting these values in equation \eqref{L3} and collecting terms with the
   same power of $x$,  \eqref{L5} is obtained.
   \end{proof}
   
   We will prove Theorem \ref{AE} by induction. The following theorem yields
   the first step of this induction.
   
   \begin{theorem}\label{first}
   \begin{equation}\label{eq_first}
   \frac{\ali(x)}{x\log x}=1+\Orden\Bigl(\frac{\log\log x}{\log x}\Bigr),
   \qquad(x\to+\infty).
   \end{equation}
   \end{theorem}
   
   \begin{proof}
   From \eqref{liexp} with $N=0$ we have   
   $\frac{\li(y)\log y}{y}=1+\Orden(\log^{-1} y)$ for $y\to\infty$.
   Since 
   $\lim_{x\to+\infty}\ali(x)=+\infty$ we may substitute $y=\ali(x)$ 
   and obtain
   \begin{equation}\label{E16}
   \frac{x\log\ali(x)}{\ali(x)}=1+\Orden\Bigl(\frac{1}{\log\ali(x)}\Bigr).
   \end{equation}
   By taking logarithms
   \begin{displaymath}
   \log x-\log\ali(x)+\log\log\ali(x)=\Orden\Bigl(\frac{1}{\log\ali(x)}\Bigr)
   \end{displaymath}
   we obtain 
   \begin{equation}\label{E12}
   \frac{\log x}{\log\ali(x)}=1+\Orden\Bigl(\frac{\log\log\ali(x)}{\log\ali(x)}
   \Bigr)
   \end{equation}
   and it follows that 
   \begin{equation}\label{E13}
   \lim_{x\to+\infty}\frac{\log x}{\log\ali(x)}=1.
   \end{equation}
   By taking log in \eqref{E12}
   \begin{displaymath}
   \log\log x-\log\log\ali(x)=\Orden\Bigl(\frac{\log\log\ali(x)}{\log\ali(x)}
   \Bigr)
   \end{displaymath}
   we obtain
   \begin{displaymath}
   \frac{\log\log x}{\log\log\ali(x)}=1+
   \Orden\Bigl(\frac{1}{\log\ali(x)}\Bigr)
   \end{displaymath}
   so that  
   \begin{equation}\label{E14}
   \lim_{x\to+\infty}\frac{\log\log x}{\log\log\ali(x)}=1.
   \end{equation}
   In view of \eqref{E13} and \eqref{E14}, we may write \eqref{E16} and \eqref{E12}
   in the form
   \begin{equation}
   \frac{x\log\ali(x)}{\ali(x)}=1+\Orden\Bigl(\frac{1}{\log x}\Bigr),
   \qquad
   \frac{\log x}{\log\ali(x)}=1+\Orden\Bigl(\frac{\log\log x}{\log x}
   \Bigr)
   \end{equation}
   and by multiplying these two, we obtain 
   \begin{displaymath}
   \frac{x\log x}{\ali(x)}=1+\Orden\Bigl(\frac{\log\log x}{\log x}\Bigr)
   \end{displaymath}
   from which \eqref{eq_first}  can  easily be deduced. 
   \end{proof}
   
   \begin{proof}[Proof of Theorem \ref{AE}]
   We proceed by induction. For $N=0$, our theorem is simply 
   Theorem \ref{first} with  $e^x$ instead of $x$. 
   
   Hence we assume \eqref{H} and try to prove the case $N+1$. 
   
   Our objective will be obtained by starting from the expansion of 
   $\li(y)$. By \eqref{liexp}
   \begin{displaymath}
   \li(y)=\frac{y}{\log y}\Bigl(1+\sum_{k=1}^{N+1} \frac{k!}{\log^k y}+
   \Orden\Bigl(\frac{1}{\log^{N+2}y}\Bigr)\Bigr).
   \end{displaymath}
   By substituting  $y=\ali(e^x)$ and applying \eqref{E13} we obtain
   \begin{equation}\label{E17}
   \frac{e^x\log\ali(e^x)}{\ali(e^x)}=1+\sum_{k=1}^{N+1} \frac{k!}{(\log \ali(e^x))^k}+
   \Orden\Bigl(\frac{1}{x^{N+2}}\Bigr).
   \end{equation}
   \medskip
   
   From our induction hypothesis, the expansion of
   $\log\ali(x)$ and $(\log\ali(x))^{-k}$ is now sought.

   By taking the log of \eqref{H} we obtain
   \begin{displaymath}
   \log\ali(e^x)=x+\log x+\log\Bigl\{1+\sum_{n=1}^N\frac{P_{n-1}(\log x)}{x^n}+
   \Orden\Bigl(\frac{\log^{N+1}x}{x^{N+1}}\Bigr)\Bigr\}.
   \end{displaymath}
   Lemma \ref{Lsubs} may be applied  with 
   \begin{displaymath}
   \log(1+X)=X-\frac{X^2}{2}+\frac{X^3}{3}-\cdots (-1)^{N+1}\frac{X^{N}}{N}+
   \Orden(X^{N+1})
   \end{displaymath}
   to obtain
   \begin{displaymath}
   \log\ali(e^x)=x+\log x+\sum_{n=1}^N\frac{Q_{n+1}(\log x)}{x^n}+
   \Orden\Bigl(\frac{\log^{N+1}x}{x^{N+1}}\Bigr).
   \end{displaymath}
   The reason why we have written 
   $Q_{n+1}$ instead 
   of $Q_n(x)$ is revealed below.
   The above may be written as 
   \begin{displaymath}
   \log\ali(e^x)=x\Bigl\{1+\frac{\log x}{x}+\sum_{n=1}^N
   \frac{Q_{n+1}(\log x)}{x^{n+1}}+
   \Orden\Bigl(\frac{\log^{N+1}x}{x^{N+2}}\Bigr)\Bigr\}
   \end{displaymath}
   or
   \begin{equation}\label{logaliexp}
   \log\ali(e^x)=x\Bigl\{1+
   \sum_{n=1}^{N+1}\frac{Q_{n}(\log x)}{x^{n}}+
   \Orden\Bigl(\frac{\log^{N+2}x}{x^{N+2}}\Bigr)\Bigr\}
   \end{equation}
   where the $Q_n(z)$ are polynomials of degree $\le n$. 
   Observe that knowing the expansion of $\ali(e^x)$ up to
   $(\log x/x)^{N+1}$ has enabled us to obtain 
   $\log\ali(e^x)$ up to $(\log x/x)^{N+2}$;
   this will be of great 
   importance in what follows. 
   
   From \eqref{logaliexp}, for all natural numbers $n$,
   \begin{displaymath}
   \frac{1}{(\log\ali(e^x))^n}=\frac{1}{x^n}\Bigl\{1+
   \sum_{k=1}^{N+1}\frac{Q_{k}(\log x)}{x^{k}}+
   \Orden\Bigl(\frac{\log^{N+2}x}{x^{N+2}}\Bigr)\Bigr\}^{-n}.
   \end{displaymath}
   By applying Lemma \ref{Lsubs} with 
   \begin{displaymath}
   (1+x)^{-n}-1=\sum_{r=1}^{N+1}\binom{-n}{r}x^r+\Orden(x^{-N-2})
   \end{displaymath}
   we obtain
   \begin{equation}\label{E19}
   \frac{1}{\{\log\ali(e^x)\}^n}=\frac{1}{x^n}
   \Bigl(1+\sum_{k=1}^{N+1}\frac{V_{n,k}(\log x)}{x^k}+
   \Orden\Bigl(\frac{\log^{N+2}x}{x^{N+2}}\Bigr)\Bigr)
   \end{equation}
   where the $V_{n,k}(z)$ are polynomials of degree $\le k$.

   By substituting these values of  $\{\log\ali(e^x)\}^{-n}$ in \eqref{E17},
   we obtain 
   \begin{displaymath}
   \frac{e^x\log\ali(e^x)}{\ali(e^x)}=1
   +\sum_{k=1}^{N+1}\frac{U_k(\log x)}{x^k}+
   \Orden\Bigl(\frac{\log^{N+2}x}{x^{N+2}}\Bigr).
   \end{displaymath}
   Hence again from \eqref{E19} with $n=1$
   \begin{multline*}
   \frac{e^x}{\ali(e^x)}=\frac{1}{x}\Bigl\{1+\sum_{k=1}^{N+1}
   \frac{V_{1,k}(\log x)}{x^k}+
   \Orden\Bigl(\frac{\log^{N+2}x}{x^{N+2}}\Bigr)\Bigr\}\times\\
   \times\Bigl\{1
   +\sum_{k=1}^{N+1}\frac{U_k(\log x)}{x^k}+
   \Orden\Bigl(\frac{\log^{N+2}x}{x^{N+2}}\Bigr)\Bigr\}
   \end{multline*}
   from which we derive that there exist polynomials $W_k(z)$ of degree
   $\le k$ such that 
   \begin{displaymath}
   \frac{xe^x}{\ali(e^x)}=1+\sum_{k=1}^{N+1}
   \frac{W_k(\log x)}{x^k}+
   \Orden\Bigl(\frac{\log^{N+2}x}{x^{N+2}}\Bigr).
   \end{displaymath}
   Another application of Lemma \ref{Lsubs}
   yields 
   \begin{displaymath}
   \frac{\ali(e^x)}{xe^x}=1+\sum_{k=1}^{N+1}
   \frac{T_k(\log x)}{x^k}+
   \Orden\Bigl(\frac{\log^{N+2}x}{x^{N+2}}\Bigr)
   \end{displaymath}
   with polynomials $T_k(z)$ of degree $\le k$. 
   Therefore, we have an asymptotic expansion of type \eqref{H} 
   with $N+1$ instead of $N$. 
   The usual argument of uniqueness of the asymptotic expansion applies here so  
   that $T_k(z)= P_k(z)$ for $1\le k\le N$. 
   \end{proof}
   
   \section{Formal Asymptotic expansion.}\label{S3}
   
   First we give some motivation. We have seen that the asymptotic
   expansion of $\ali(e^x)$ is  
   \begin{displaymath}
   \ali(e^x)=xe^xV(x,\log x),\quad\text{where}\quad
    V(x,y):=1+\sum_{n=1}^\infty \frac{P_{n-1}(y)}{x^{n}}
   \end{displaymath}
   and differentiation yields
   \begin{displaymath}
   e^x\log \ali(e^x)  = (e^x +xe^x)V+xe^x V_x+e^x V_y.
   \end{displaymath}
   Here $\log\ali(e^x)=\log\bigl(xe^xV(x,\log x)\bigr)=y+x+\log V$, so that
   \begin{displaymath}
   y+x+\log V=V+xV+xV_x+V_y
   \end{displaymath}
   which we  write as
   \begin{equation}\label{E}
   V=1+\frac{y}{x}-\frac{1}{x}V-V_x-\frac{1}{x}V_y+\frac{1}{x}\log V.
   \end{equation}
   This ends our motivation for considering this equation.
   
   Consider now the ring $A$ of the formal power series of the type
   \begin{displaymath}
   \sum_{n=0}^\infty\frac{q_n(y)}{x^n}
   \end{displaymath}
   where the $q_n(y)$ are polynomials with complex coefficients of  degree less than 
   or equal to $n$. In particular $q_0(y)$ is a constant.
   
   It is clear that $A$, with the obvious operations, is a ring. The elements with
   $q_0=0$ form a maximal ideal $I$. An element $1+u$ with $q_0=1$ is invertible,
   with inverse $1-u+u^2-\cdots$. It follows that if $a\not\in I$, then $a$ is also
   invertible. Hence $I$ is the unique maximal ideal and $A$ is a local ring.
   If $a\in A$ is a non-vanishing element, then there exists a least natural number $n$
   with $q_n(y)\ne0$. We define $\deg(a)=n$ in this case, with $\deg(0)=\infty$.
   
   As is usual in local rings, (see \cite{La}) we may define a 
   topology induced by the norm 
   $\Vert a\Vert = 2^{-\deg(a)}$, which,   with the associated metric, 
   induces a complete 
   metric space.  Indeed $A$ is isomorphic to $\C[[X,Y]]$, by means of 
   the application that
   sends $X\mapsto x^{-1}$, $Y\mapsto yx^{-1}$.
   
   Given $a\in A$ with $a=\sum_{n=0}^\infty\frac{q_n(y)}{x^n}$, we define 
   two derivates
   \begin{displaymath}
   a_x = -\sum_{n=1}^\infty \frac{nq_n(y)}{x^{n+1}}\quad\text{and}\quad
   a_y=\sum_{n=1}^\infty \frac{q'_n(y)}{x^n}.
   \end{displaymath}
   
   Finally the set $U\subset A$ of elements with $q_0=1$ form a multiplicative
   subgroup of $A^*$ (the group of invertible elements of $A$). 
   For $1+u\in U$,  we define
   \begin{displaymath}
   \log(1+u)=\sum_{k=1}^\infty (-1)^{k+1}\frac{u^k}{k}
   \end{displaymath}
   which is a series that  is easily shown to converge since $u^k\in I^k$. 
   
   We are now ready to  prove the following
   \begin{theorem}
   The equation \eqref{E} has one and only one solution in the 
   ring $A$.
   \end{theorem}

   \begin{proof} For $V\in U$,  we define $T(V)$ as
   \begin{displaymath}
   T(V):=1+\frac{y}{x}-\frac{1}{x}V-V_x-\frac{1}{x}V_y+\frac{1}{x}\log V.
   \end{displaymath}
   It is clear that $T(V)\in U$. We may apply Banach's  fixed-point theorem. 
   Indeed, we have 
   $\deg(T(V)-T(W))\le 1+\deg(V-W)$, so that $\Vert T(V)-T(W)\Vert\le \frac12
   \Vert V-W\Vert$. 
   
   By Banach's theorem there is a unique solution to $V=T(V)$. We may obtain 
   this solution  as the limit of the sequence $T^n(1)$. In fact, since
   $\deg(T(V)-T(W))\le 1+\deg(V-W)$,  in each iteration
   we obtain one further term of the expansion.
   In this way, it is easy to prove that the solution is 
   \begin{displaymath}
   V=1+\frac{y-1}{x}+\frac{y-2}{x^2}+\cdots
   \end{displaymath}
   However, we are going to find more direct methods to compute the terms of the 
   expansion.
   \end{proof}
   
   \begin{definition}\label{defP}
   Let $V$ be the unique solution to equation \eqref{E}. Since it is in $A$,
   it has the form
   \begin{equation}\label{defV}
   V(x,y)=1+\sum_{n=1}^\infty \frac{P_{n-1}(y)}{x^n}
   \end{equation}
   where for $n\ge0$,   $P_{n-1}(y)$ is a  polynomial of degree $\le n$. 
   \end{definition}
   
   In the following sections, we prove that $V$ yields the asymptotic expansion of
   $\ali(e^x)$. For this proof the following property is crucial.
   
   \begin{theorem}\label{cancel}
   For $N\ge1$, let 
   \begin{equation}
   W(x,y)= W_N(x,y):=1 + \sum_{k=1}^N \frac{P_{n-1}(y)}{x^n}.
   \end{equation}
   Then 
   \begin{equation}\label{solapprox}
   W-1-\frac{y}{x}+\frac{1}{x}W+W_x+\frac{1}{x}W_y-\frac{1}{x}\log W=
   -\frac{P_{N}(y)}{x^{N+1}}+\frac{u(y)}{x^{N+2}}
   +\frac{v(y)}{x^{N+3}}+\cdots
   \end{equation}
   \end{theorem}
   
   \begin{proof}
   By the definition of the $P_n$ we know that $\deg(V-W)\ge N+1$.  
   Therefore, $\deg(V-T(W))\ge N+2$.
   That is
   \begin{displaymath}
   V-T(W)=\frac{u_0(y)}{x^{N+2}}+\frac{v_0(y)}{x^{N+3}}+\cdots 
   \end{displaymath}
   where $u_0$, $v_0$ are polynomials. We also have
   \begin{displaymath}
   V=W+\sum_{n=N+1}^\infty \frac{P_{n-1}(y)}{x^n}
   \end{displaymath}
   so that 
   \begin{displaymath}
   W-T(W)=\frac{u_0(y)}{x^{N+2}}+\frac{v_0(y)}{x^{N+3}}+\cdots -
   \sum_{n=N+1}^\infty \frac{P_{n-1}}{x^n}.
   \end{displaymath}
   That is,
   \begin{displaymath}
   W-T(W)=-\frac{P_{N}(y)}{x^{N+1}}+\frac{u(y)}{x^{N+2}}+\frac{v(y)}{x^{N+3}}+\cdots
   \end{displaymath}
   for certain polynomials $u$, $v$, \dots
   \end{proof}

   In the sequel $V$ will  denote the unique solution to \eqref{E}. The element
   $\log V$ belongs to $A$, so that there are polynomials $Q_n(y)$ of
   degree less than or equal to $n$  such
   that 
   \begin{equation}\label{defQ}
   \log V = \sum_{n=1}^\infty \frac{Q_n(y)}{x^n}.
   \end{equation}
   From equation \eqref{E}, we may obtain $\log V$ in terms of  $V$ and its
   derivatives. It is easy to obtain from this expression the following
   relation
   \begin{equation}\label{eqQ}
   Q_n(y)=P_n(y)-(n-1) P_{n-1}(y)+P'_{n-1}(y),\qquad (n\ge1).
   \end{equation}
  
   \begin{theorem}
   The polynomials $P_n(y)$ that appear in the unique solution 
   \eqref{defV} to equation
   \eqref{E}  may be computed by the following recurrence relations:
   \begin{equation}\label{rec}
   \begin{split}
   P_0=y-1,\quad \text{and for $n\ge1$}\mskip150mu \\ P_n=n P_{n-1}-P'_{n-1}+\frac{1}{n}
   \sum_{k=1}^{n-1}k\bigl\{(k-1)P_{k-1}-P_{k}-P'_{k-1}\bigr\}P_{n-k-1}.
   \end{split}\raisetag{-16pt}
   \end{equation}
   \end{theorem}
   
   \begin{proof}
   By differentiating \eqref{defQ} with respect to $x$, we obtain
   \begin{displaymath}
   \Bigl(\sum_{n=1}^\infty \frac{nQ_n(y)}{x^{n+1}}\Bigr)
   \Bigl(1+\sum_{n=1}^\infty \frac{P_{n-1}(y)}{x^n}\Bigr)=
   \sum_{n=1}^\infty \frac{nP_{n-1}(y)}{x^{n+1}}.
   \end{displaymath}
   By equating the coefficients of $x^{-n-1}$, we obtain
   \begin{equation}\label{dif1}
   nP_{n-1}=nQ_n+\sum_{k=1}^{n-1}kQ_k P_{n-k-1},\qquad (n\ge2).
   \end{equation}
   Now we substitute the values of the $Q_n$ given in \eqref{eqQ}
   \begin{displaymath}
   nP_{n-1}=nP_{n}-n(n-1)P_{n-1}+nP'_{n-1}+\sum_{k=1}^{n-1}k
   Q_k P_{n-k-1}
   \end{displaymath}
   so that
   \begin{displaymath}
   nP_n=n^2P_{n-1}-nP'_{n-1}-\sum_{k=1}^{n-1}k
   \bigl\{P_k-(k-1)P_{k-1}+P'_{k-1}\bigr\} P_{n-k-1}.
   \end{displaymath}   
   \end{proof}
   
   From this expression it is very easy to compute the first terms of the expansions
   \begin{multline*}
   V=1+\frac{y-1}{x}+\frac{y-2}{x^2}-
   \frac{y^2-6y+11}{2x^3}+\frac{2y^3-21y^2+84y-131}{6x^4}-\\
   -\frac{3y^4-46y^3+294y^2-954y+1333}{12x^5}+
   \cdots,
   \end{multline*}
   \begin{multline*}
   \log V = \frac{y-1}{x}-\frac{y^2-4y+5}{2x^2}+\frac{2y^3
   -15y^2+42y-47}{6x^3}-\\
   -\frac{3y^4-34y^3+156y^2-366y+379}{12x^4}+\cdots
   \end{multline*}

   \begin{theorem}\label{integer}
   We have
   
   \begin{itemize} 
   
   \item[(a)] For $n\ge1$, the degree of $P_n$ is less than or equal to $n$.
   
   \item[(b)] $n! \,P_n(y)$ has integer coefficients.
   \end{itemize}
   \end{theorem}
   
   \begin{proof}
   The equation \eqref{E} may be written
   \begin{displaymath}
   V-1-\frac{y}{x}=\frac{1}{x}\bigl(\log V-V-x V_x-V_y\bigr).
   \end{displaymath}
   Since $xV_x\in A$, it is clear that 
   \begin{displaymath}
   x V-x-y= -1+\sum_{n=1}^\infty \frac{P_n(y)}{x^n}  \in A.
   \end{displaymath}
   This implies that the degree of $P_n$ is less than or equal to $n$.
   
   We prove (b) by induction. The first few $P_n$ satisfy this property.  
   We define  $p_k:=k!P_k$ so that
   the recurrence relation \eqref{rec}  may be written as    
   \begin{multline*}
   p_n=n^2p_{n-1}-np'_{n-1}+\\+(n-1)\sum_{k=1}^{n-1}\binom{n-2}{k-1}
   \bigl\{k(k-1)p_{k-1}-p_k-kp'_{k-1}\bigr\} p_{n-k-1}.
   \end{multline*}
   Hence, by induction,  all  $p_n$ have integer coefficients.
   \end{proof}

   The most significant contribution by Cipolla is his proof of a recurrence
   for the coefficients $a_{n,k}$ of $P_{n}$ (see \eqref{formP}), 
   which is better than the recurrence given in \eqref{rec}.
   We intend  to give a slightly different proof. The result of Cipolla 
   is equivalent
   to the following surprising fact: The solution $V$ of equation \eqref{E}
   formally satisfies the following linear partial differential equation:
   \begin{equation}\label{PDE}
   V=(x-1) V_y-x V_x.
   \end{equation}
   This equation can  easily be deduced  from the following Theorem.

   \begin{theorem}
   For $n\ge1$, we have
   \begin{equation}\label{Cipolla}
   \begin{aligned}
   (n-1)P_{n-1}(y)&=P'_{n-1}(y)-P'_n(y),\qquad (n\ge1)\\
   (n-1)Q_{n-1}(y)&=Q'_{n-1}(y)-Q'_n(y),\qquad (n\ge2).
   \end{aligned}
   \end{equation}
   \end{theorem}
   
   \begin{proof}
   We will proceed by induction. For $n\le 3$ it can be verified that 
   these equalities are satisfied. 
   
   We now assume that \eqref{Cipolla} is satisfied for $n\le N$, and 
   we will show that these equations are true for $n=N+1$. 
   
   By differentiating \eqref{defQ} with respect to $y$ we get
   \begin{displaymath}
   \Bigl(\sum_{n=1}^\infty \frac{Q'_n(y)}{x^n}\Bigr)
   \Bigl(1+\sum_{n=1}^\infty \frac{P_{n-1}(y)}{x^n}\Bigr)=
   \sum_{n=1}^\infty \frac{P'_{n-1}(y)}{x^n}
   \end{displaymath}
   so that by equating the coefficients of $x^{-N-1}$ and of $x^{-N}$  we obtain
   \begin{displaymath}
   Q'_{N+1}=P'_N-\sum_{k=0}^{N-1}P_k Q'_{N-k},\qquad
   Q'_{N}=P'_{N-1}-\sum_{k=0}^{N-2}P_k Q'_{N-k-1}.
   \end{displaymath}
   Subtracting  
   these equations we get
   \begin{displaymath}
   Q'_{N+1}-Q'_N = P'_N-P'_{N-1}-\sum_{k=0}^{N-2}P_k(Q'_{N-k}-Q'_{N-k-1})-P_{N-1}
   \end{displaymath}
   and by the induction hypothesis this is equal to 
   \begin{multline*}
   -(N-1)P_{N-1}+\sum_{k=0}^{N-2}P_k\cdot(N-k-1)Q_{N-k-1} -P_{N-1}=\\
   =-NP_{N-1}+\sum_{k=1}^{N-1}kQ_kP_{N-k-1}.
   \end{multline*}
   By \eqref{dif1} this is equal to $NP_{N-1}-NQ_N$ so that we obtain
   \begin{displaymath}
   Q'_{N+1}-Q'_N =-N Q_N.
   \end{displaymath}
   This is the second equation of \eqref{Cipolla} for $n=N+1$. 
   In order to achieve the
   result for the first equation,  observe that from \eqref{eqQ} we get
   \begin{align*}
   NQ_N &= NP_N-N(N-1)P_{N-1}+NP'_{N-1}\\
   -Q'_N&= -P'_N+(N-1)P'_{N-1}-P''_{N-1}\\
   Q'_{N+1}&= P'_{N+1}-N P'_N+P''_N.
   \end{align*}
   By adding these equations we obtain
   \begin{multline*}
   0 = NP_N-P'_N+P'_{N+1}+N\bigl\{P'_{N-1}-P'_N-(N-1)P_{N-1}\bigr\}-\\-
   \bigl\{P''_{N-1}-P''_N-(N-1)P'_{N-1}\bigr\}=NP_N-P'_N+P'_{N+1}
   \end{multline*}
   which is the first equation of \eqref{Cipolla} for $n=N+1$. 
   \end{proof}
   
   We define the coefficients $a_{n,k}$ implicitly by 
   \begin{multline}\label{formP}
   P_n(y)=\frac{(-1)^{n+1}}{n!}\Bigl(a_{n,0}y^n-a_{n,1}y^{n-1}+\cdots
   (-1)^n a_{n,n}\Bigr)=\\
   =\frac{(-1)^{n+1}}{n!}\sum_{k=0}^n (-1)^k\, a_{n,k}\,y^{n-k},\qquad (n\ge1).
   \end{multline}
   Analogously, $Q_n$ is of a degree less than or equal to $n$, and we define 
   the coefficients $b_{n,k}$ implicitly by
   \begin{equation}
   Q_n(y)=\frac{(-1)^{n+1}}{n!}\sum_{k=0}^n (-1)^k\, b_{n,k}\,y^{n-k},\qquad (n\ge1).
   \end{equation}  
   
   \begin{remark}
   $P_0(y)$ has degree $1$, which is not given by \eqref{formP}. 
   However, we can extend
   the definition of $a(n,k)$ in such a way that, for $n\ge1$ we have
   $a(n,k)=0$ for $k<0$ or $k>n$. Then a formula such as \eqref{formP} also 
   holds for $n=0$ if we add up the values from $k=-1$ to $k=n$ and define 
   $a(0,0)=1$, $a(0,-1)=1$ and $a(0,k)=0$ for other values of $k$
   \begin{equation}\tag{30 bis}
   P_n(y)
   =\frac{(-1)^{n+1}}{n!}\sum_{k=-1}^n (-1)^k\, a_{n,k}\,y^{n-k},\qquad (n\ge0).
   \end{equation}
   Note that $Q_0(y)$  remains undefined.
   \end{remark} 
   
   \begin{theorem}\label{TCipolla}
   For $1\le n$ and  $0\le k<n$, we have (when defined)
   \begin{equation}\label{rec-ab}
   a_{n,k}=na_{n-1,k-1}+\frac{n(n-1)}{n-k}a_{n-1,k},\quad
   b_{n,k}=nb_{n-1,k-1}+\frac{n(n-1)}{n-k}b_{n-1,k}.
   \end{equation}
   For $1\le n$ and  $0\le k\le n$, we have
   \begin{equation}\label{mixed}
   a_{n,k}=b_{n,k}+(n-k+1)a_{n,k-1}.
   \end{equation}
   For $n\ge1$, we have
   \begin{equation}\label{endcoef}
   b_{n,n}=n a_{n-1,n-1}+\sum_{k=1}^{n-1}\binom{n-1}{k}\; k\; b_{k,k}\,a_{n-k-1,n-k-1}.
   \end{equation}
   \end{theorem}
   
   \begin{proof}
   \eqref{rec-ab} is obtained by equating the coefficients of $y^{n-k-1}$
   in the first equation in \eqref{Cipolla}. In this way, we obtain
   \begin{multline*}
   (n-1)\frac{(-1)^{n+k}}{(n-1)!}a_{n-1,k}=\\=
   (n-k)\frac{(-1)^{n+k+1}}{(n-1)!}a_{n-1,k-1}-(n-k)\frac{(-1)^{n+k+1}}{n!}
   a_{n,k}.
   \end{multline*}
   If $n\ne k$, then the equation for $a_{n,k}$ in \eqref{rec-ab} is obtained.
   The other equation in $b_{n,k}$ is obtained analogously from the second
   equation in \eqref{Cipolla}.
   
   To prove \eqref{mixed}, observe that by \eqref{eqQ},  
   $Q_n=P_n-(n-1) P_{n-1}+P'_{n-1}$, 
   and from \eqref{Cipolla} it follows that 
   \begin{equation}\label{relQP}
   Q_n=P_n+P'_n,\qquad (n\ge1).
   \end{equation}
   Now by equating  the coefficient of $y^{n-k}$ in both members of this equality
   we obtain \eqref{mixed}.
   
   Finally \eqref{endcoef} follows from \eqref{dif1}. 
   Recall that $-\frac{a_{n,n}}{n!}$ and $-\frac{b_{n,n}}{n!}$ are
   respectively  the values of $P_n(0)$, and $Q_n(0)$. Hence, by setting 
   $y=0$ in \eqref{dif1}, we obtain  \eqref{endcoef}  through 
   multiplication  by $(n-1)!$ and the reordering of the terms.
   \end{proof}

   The main problem now is that equations \eqref{rec-ab}
   do not allow us to compute the coefficients  $a_{n,n}$.
   Cipolla gives an algorithm to simultaneously compute the coefficients 
   $a_{n,k}$ and $b_{n,k}$ based on  Theorem \ref{TCipolla}. 
   In the procedure of 
   Cipolla, these key  coefficients $a_{n,n}$ are recursively computed 
   using all the previous coefficients.  We prefer a method that computes 
   $A_n := a_{n,n}$ and $B_n:=a_{n,n-1}$ separately  and then compute
   the remaining coefficients by  using \eqref{rec-ab}.

   \begin{theorem}\label{T:algorithm}
   In order to compute the numbers $a_{n,k}$,  we may first compute the sequences
   $A_n:=a_{n,n}$ and $B_n:=a_{n,n-1}$ by the recursions
   \begin{gather}
   A_0=1,\quad A_1=2,\quad B_0=1,\quad B_1=1,\\
   B_n=n B_{n-1}+n(n-1) A_{n-1} \label{Alg1}
   \end{gather}
   \begin{multline}\label{Alg2}
   A_n=n^2 A_{n-1}+nB_{n-1}-\\-(n-1)\sum_{k=1}^{n-1}\binom{n-2}{k-1}
   \bigl\{k(k-1)A_{k-1}-A_k+k B_{k-1}\bigr\} A_{n-k-1}.
   \end{multline}
   After this one 
   we may  obtain $a(n,k):=a_{n,k}$. Setting
   \begin{displaymath} 
   a(0,0)=1,\quad a(0,-1)=1, \quad a(1,0)=1, \quad a(1,1)=2
   \end{displaymath}
   and all other $a(0,k)$ and $a(1,k)=0$. Then,  for $n\ge2$, put
   \begin{align}\label{Alg3}
   a(n,n)&= A_n,\notag\\
    a(n,k) &= n a(n-1,k-1)+\frac{n(n-1)}{n-k}a(n-1,k), \quad
   (0\le k<n)
   \end{align}
   where   $a(n,k)=0$ for $k<0$ or $k>n$.
   
   Finally, we may obtain the $b(n,k):=b_{n,k}$ from 
   \begin{equation}\label{Alg4}
   b(n,k) = a(n,k)-(n-k+1)a(n,k-1).
   \end{equation}
   \end{theorem}

   \begin{proof}
   The  constant term of $P_n$ is $-\frac{A_n}{n!}$ 
   and the coefficient of $y$ in $P_n$ is 
   $\frac{B_n}{n!}$, so that
   equation \eqref{Alg1} follows from the first equation in  
   \eqref{Cipolla} taking it with $y=0$.  
   
   In the same way, \eqref{Alg2} follows from \eqref{rec}, 
   and \eqref{Alg3} is the first equation in \eqref{rec-ab}.

   Equation  \eqref{Alg4} for the $b(n,k)$  follows easily from \eqref{relQP}.
   \end{proof}   
   
   The array of  coefficients $a(n,k)$ for $0\le n, k\le 7$,   reads
   \bigskip
   
   {\small
   \begin{tabular}{rrrrrrrrr}
   1 & 0  &  0  & 0  &  0  &  0  &  0  &  0 \\ 
   1 & 2  &  0  & 0  &  0  &  0  &  0  &  0 \\ 
   1 & 6  & 11  & 0  &  0  &  0  &  0  &  0 \\
   2 & 21 & 84 & 131 &  0  &  0  &  0  &  0 \\
   6 & 92 & 588 & 1908 & 2666 &  0  &  0  &  0 \\
   24 & 490 & 4380 & 22020 & 62860 & 81534 &  0  &  0 \\
   120 & 3084 & 35790 & 246480 & 1075020 & 2823180 & 3478014 & 0 \\
   720 & 22428 & 322224 & 2838570 & 16775640 & 66811920 & 165838848 &
     196993194
   \end{tabular}
   } 
   \bigskip
   
   and the $b(n,k)$ for $1\le n\le 7$ and $0\le k\le 7$ are
   \bigskip
   
   {\small
   \begin{tabular}{rrrrrrrrr}
   1 & 1  &  0  & 0  &  0  &  0  &  0  &  0 \\ 
   1 & 4  &  5  & 0  &  0  &  0  &  0  &  0 \\ 
   2 & 15  & 42  & 47  &  0  &  0  &  0  &  0 \\
   6 & 68 & 312 & 732 &  758  &  0  &  0  &  0 \\
   24 & 370 & 2420 & 8880 & 18820 &  18674  &  0  &  0 \\
   120 & 2364 & 20370 & 103320 & 335580 & 673140 &  654834  &  0 \\
   720 & 17388 & 187656 & 1227450 & 5421360 & 16485000 & 32215008 & 31154346 \\
   \end{tabular}
   } 
   \bigskip

   \begin{theorem}

   \begin{itemize}
   \item[(a)]
   The coefficients $b(n,k)$ are integers.
   
   \item[(b)] $a(n,k)\ge0$ and $b(n,k)\ge0$.
   
   \item[(c)] $a(n,k-1)\le a(n,k)$ for $1\le k\le n$. 
   
   \item[(d)] For $n\ge1$, $a(n,0) = (n-1)!$. 
   
   \end{itemize}
   \end{theorem}
   
   \begin{proof}
   (a) We have proved in Theorem  \ref{integer} that  the numbers 
   $a(n,k)$ are integers,
   so that from \eqref{Alg4}, the coefficients $b(n,k)$ are also integers. 
   
   (b) We proceed by induction on $n$. Assuming that we have proved 
   that $a(m,k)$ and $b(m,k)$ are positive for $m<n$, it follows 
   from \eqref{rec-ab}
   that $a(n,k)$ and $b(n,k)$ are positive for $0\le k<n$.  Then 
   \eqref{endcoef} implies that $b(n,n)\ge0$, and \eqref{mixed} with $k=n$ 
   proves that $a(n,n)\ge0$. 
   
   (c) This is a simple consequence of \eqref{mixed}. 
   
   (d) The equation follows from \eqref{Alg3} by induction. 
   \end{proof}

   \begin{theorem}\label{T:complexity}
   By means of the rule in Theorem \ref{T:algorithm}, one may 
   compute all  coefficients
   $a_{n,k}$ of the polynomials $P_n(y)$ for $1\le n\le N$  in $\Orden(N^2)$
   coefficient operations.
   \end{theorem}
   
   \begin{proof}
   We count the operations needed, following the indications in 
   Theorem \ref{T:algorithm},
   to compute every $a_{n,k}$ for $0\le n\le N$ and $0\le k\le n$. 

   First we must compute the  numbers $\binom{m}{j}$ for $0\le m\le N-2$.
   Using the scheme of the usual triangle, we need to carry out 
   $\sum_{k=1}^{N-3}k$ additions, which 
   involves $(N-2)(N-3)/2$ operations. 

   The numbers $B_n$ must now be computed for $2\le n\le N$ by means of the formula
   \begin{displaymath}
   B_n=n*(B_{n-1}+(n-1)*A_{n-1}).
   \end{displaymath}
   Each $B_n$ requires $4$ operations, therefore a total of $4(N-1)$ operations
   are needed. 
   We  compute the $A_n$ for $2\le n\le N$ using the formula
   \begin{multline*}
   A_n=n*n*A_{n-1}+n*B_{n-1}-(n-1)*\\
   *\sum_{k=1}^{n-1}\binom{n-2}{k-1}*\{
   k*(k-1)*A_{k-1}-A_k+k*B_{k-1}\}*A_{n-k-1}.
   \end{multline*}
   Hence  $A_n$ requires
   $7+\sum_{k=1}^{n-1}8=8n-1$ operations.
   All $A_n$ together take  $\sum_{n=2}^N (8n-1)=4N^2+3N-7$ operations. 
   These numbers are the $a_{n,n}$.  The $a_{0,k}$ and $a_{1,k}$ require 
   no operations. Finally we compute for $0\le k<n$ 
   \begin{displaymath}
   a_{n,k}=n*\bigl\{a_{n-1,k-1}+(n-1)*a_{n-1,k}/(n-k)\bigr\}.
   \end{displaymath}
   Therefore, $a_{n,k}$ takes $6$ operations. For each $n$, every $a_{n,k}$ for 
   $1\le k<n$ takes $6(n-1)$ operations. And each $a_{n,k}$ for $2\le n\le N$ takes
   $\sum_{n=2}^N6(n-1)=3N(N-1)$. 

   The total cost in number of operations is therefore
   \begin{multline*}
   \frac{(N-2)(N-3)}{2}+4(N-1)+ 4N^2+3N-7+3N(N-1)=\\
   =\frac12(15N^2+3N-16).
   \end{multline*}   
   \end{proof}
   
   \section{Bounds for the asymptotic expansion.}\label{Sbounds}
   
   \subsection{The sequence $(a_n)$.}
   
   First we define a sequence of numbers as the coefficients of a formal 
   expansion in $A$.

   \begin{lemma}
   There exists a sequence of integers $(a_n)$ such that
   \begin{equation}\label{E26}
   \log\Bigl(1-\sum_{n=1}^\infty \frac{n!}{x^n}\Bigr)^{-1}=
   \sum_{n=1}^\infty \frac{a_n}{n}\frac{1}{x^n}.
   \end{equation}
   The coefficients may be computed by the recursion
   \begin{equation}\label{Ereca}
   a_1=1, \quad a_n =  n!\cdot n+ \sum_{k=1}^{n-1}k!\, a_{n-k}.
   \end{equation}
   \end{lemma}
   
   \begin{proof}
   It is clear that $u=1-\sum n! x^{-n}\in U\subset A$, so that
   $u^{-1}\in U$ and $\log u^{-1}$ are well defined. To obtain the recursion
   we differentiate \eqref{E26} to obtain 
   \begin{displaymath}
   -\sum_{n=1}^\infty \frac{n\cdot n!}{x^{n+1}}=-\Bigl( 
   \sum_{n=1}^\infty \frac{a_n}{x^{n+1}}\Bigr)
   \Bigl(1-\sum_{n=1}^\infty \frac{n!}{x^n}\Bigr).
   \end{displaymath}
   Equation \eqref{Ereca} is obtained by equating the coefficients of $x^{-n-1}$.
   The recurrence \eqref{Ereca} proves that  $a_n$ is a natural number for 
   each $n\ge1$.
   \end{proof}
   
   The first terms of the sequence $(a_n)_{n=1}^\infty $ are
   \begin{displaymath}
   1, 5, 25, 137, 841, 5825, 45529, 399713, 3911785, 42302225,\dots
   \end{displaymath}
   
   \begin{lemma}
   For each natural number $n$ we have
   \begin{equation}\label{boundan}
   a_n\le 2n\cdot n!.
   \end{equation}
   \end{lemma}
   
   \begin{proof}
   We may verify this property for $a_1$, $a_2$, $a_3$ and $a_4$ directly. 
   For $n\ge 4$
   we proceed by induction.  Assume the inequality for $a_k$ with $k<n$, so
   that by \eqref{Ereca}
   \begin{multline*}
   1\le \frac{a_n}{n!\cdot n}\le 1+
   \sum_{k=1}^{n-1}\frac{a_{n-k}}{(n-k)!\cdot(n-k)}
   \frac{n-k}{n}\binom{n}{k}^{-1}\le\\
   \le 1+2\Bigl(\frac{1}{n}+\sum_{k=2}^{n-2}\binom{n}{k}^{-1}+ 
   \frac{1}{n^2}\Bigr) \le 1+2\Bigl(\frac{1}{n}+\frac{1}{n^2}+
   (n-3)\frac{2}{n(n-1)}\Bigr).
   \end{multline*}
   For  $n\ge 4$, it is easy to see that this is  $\le 2$. 
   \end{proof}

   \begin{lemma}\label{Lsimple}
   For each natural number $N$ there is a positive constant $c_N$ such that 
   \begin{equation}\label{simpleineq}
   x\Bigl(1-\sum_{n=1}^N \frac{n!}{x^n}\Bigr)\ge 1,\qquad  x\ge c_N.
   \end{equation}
   \end{lemma}
   
   \begin{proof}
   It is clear that the left-hand side of \eqref{simpleineq} is increasing
   and  tends to
   $+\infty$ when $x\to+\infty$, from which the existence of $c_N$ is clear.
   
   The value of $c_N$ may be determined as the solution of the equation 
   \begin{equation}
   x\Bigl(1-\sum_{n=1}^N \frac{n!}{x^n}\Bigr)=1,\qquad  x>1.
   \end{equation}
   In this way we found the following values.
   \bigskip
   
   \begin{tabular}{|c|l||c|l||c|l||c|l|} \hline
   $c_1$ & 2 & $c_6$ & 4.15213 & $c_{11}$ & 5.61664 & $c_{20}$ & \ 8.70335\\
   \hline
   $c_2$ & 2.73205 & $c_7$ & 4.43119 & $c_{12}$ & 5.93649 & $c_{30}$ & 12.34925\\
   \hline
   $c_3$ & 3.20701 & $c_8$ & 4.71412 & $c_{13}$ & 6.26449 & $c_{40}$ & 16.03475\\
   \hline
   $c_4$ & 3.56383 & $c_9$ & 5.00517 & $c_{14}$ & 6.59947 & $c_{50}$ & 19.72833\\
   \hline
   $c_5$ & 3.86841 & $c_{10}$ & 5.30597 & $c_{15}$ & 6.94035 & $c_{60}$ & 23.42351\\
   \hline
   \end{tabular}
   \bigskip
   \end{proof}
   
   \begin{remark} \label{remark54}
   Notice that for $x\ge c_N$ the sum in \eqref{simpleineq} is positive and less 
   than $1$. 
   \end{remark}
   
   \begin{proposition}\label{usean}
   For each natural number $N$ there exists $d_N>0$ such that,  for $x\in\C$ 
   with $|x|\ge d_N$, there exists $\theta$ with $|\theta|\le1$ such that
   \begin{equation}
   \log\Bigl(1-\sum_{n=1}^N \frac{n!}{x^n}\Bigr)^{-1}=
   \sum_{n=1}^N \frac{a_n}{n}\frac{1}{x^n}+
   \theta \frac{a_{N+1}}{N+1}\frac{1}{x^{N+1}},
   \qquad  |x|>d_N.
   \end{equation}
   \end{proposition}
   
   \begin{proof} By comparing the  expansions \eqref{E26} and 
   \begin{equation}\label{apartial}
   \log\Bigl(1-\sum_{n=1}^N \frac{n!}{x^n}\Bigr)^{-1}=
   \sum_{n=1}^N \frac{a_n}{n}\frac{1}{x^n}+\sum_{n=N+1}^\infty 
   \frac{b_n}{n}\frac{1}{x^n}
   \end{equation} it is clear that $\frac{b_{N+1}}{N+1}+(N+1)!=\frac{a_{N+1}}{N+1}$,  
   so that
   $b_{N+1}<a_{N+1}$. 
   
   The above expansion is convergent for all sufficiently large $|x|$, so that
   \begin{displaymath}
   \sum_{n=N+1}^\infty \frac{b_n}{n}\frac{1}{x^n}=
   \frac{b_{N+1}}{N+1}\frac{1}{x^{N+1}}g_N(x)
   \end{displaymath}
   where $\lim_{x\to\infty} g_N(x)=1$. Hence there exist sufficiently large 
   $d_N$ 
   such that 
   \begin{displaymath}
   |b_{N+1}g_N(x)|<a_{N+1},\qquad |x|>d_N.
   \end{displaymath}
   This ends the proof of the existence of 
   $d_N$.  
   
   We have 
   \begin{multline*}
   \Bigl(\log\Bigl(1-\sum_{n=1}^N \frac{n!}{x^n}\Bigr)^{-1}-
   \sum_{n=1}^N \frac{a_n}{n}\frac{1}{x^n}\Bigr)
   \frac{(N+1)x^{N+1}}{a_{N+1}}=\\=\frac{(N+1)}{a_{N+1}}
   \sum_{n=N+1}^\infty \frac{b_n}{n}\frac{1}{x^{n-N-1}}.
   \end{multline*}   
   Since all  $a_n$ and $b_n$ are positive, this is a decreasing
   function for $x\to+\infty$, and the lowest value of $d_N$ 
   will be the unique solution of 
   \begin{displaymath}
   \Bigl(\log\Bigl(1-\sum_{n=1}^N \frac{n!}{x^n}\Bigr)^{-1}-
   \sum_{n=1}^N \frac{a_n}{n}\frac{1}{x^n}\Bigr)
   \frac{(N+1)x^{N+1}}{a_{N+1}}=1.
   \end{displaymath}
   We obtain the following table of values
   \bigskip
   
   \begin{tabular}{|c|l||c|l||c|l||c|l|} \hline
   $d_1$ & 1.03922 & $d_6$ & 4.54145 & $d_{11}$ & 5.73661 & $d_{20}$ & \ 8.73298\\
   \hline
   $d_2$ & 2.38568 & $d_7$ & 4.75734 & $d_{12}$ & 6.03061 & $d_{30}$ & 12.37349\\
   \hline
   $d_3$ & 3.33232 & $d_8$ & 4.97336 & $d_{13}$ & 6.33969 & $d_{40}$ & 16.05983\\
   \hline
   $d_4$ & 3.92171 & $d_9$ & 5.20626 & $d_{14}$ & 6.66091 & $d_{50}$ & 19.75448\\
   \hline
   $d_5$ & 4.28707 & $d_{10}$ & 5.46090 & $d_{15}$ & 6.99175 & $d_{60}$ & 23.45053\\
   \hline
   \end{tabular}
   \bigskip
   \end{proof}
   
   \begin{remark}
   The numbers $d_N$ in Lemma \ref{usean} are very similar to the numbers $c_N$
   of Lemma \ref{Lsimple}. This is no more than an experimental observation, but 
   since the $c_N$ numbers  are easy to compute and  $d_N$ are somewhat elusive,
   it has been useful to start from $c_N$ as an approximation to $d_N$
   in order to compute $d_N$.
   \end{remark}
   
   \subsection{Some inequalities.}
   
   \begin{lemma}\label{comparation}
   For $u\ge2$ we have $\log\ali(u)\le 2\log u$. For $u\ge e^2$ we have 
   $\ali(u)\le 2u\log u$.
   \end{lemma}
   
   \begin{proof}
   The first inequality is equivalent to $\ali(u)\le u^2$. Since $\li(x)$ is strictly 
   increasing, the inequality is equivalent to $u\le \li(u^2)$.
   
   For $u>2$ we have $\li(u)>\li(2)=1.04516\dots$ so that
   \begin{displaymath}
   \li(u^2)=\li(u)+\int_u^{u^2}\frac{dt}{\log t}>1+\frac{u^2-u}{\log u^2}.
   \end{displaymath}
   Hence, our inequality follows from $\frac{u^2-u}{\log u^2}>u-1$, that is from 
   $u>2\log u$. However, this last inequality is certainly true for $u>2$.
   
   The second inequality is equivalent to $u\le\li(2u\log u)$ and has a similar 
   easy proof.
   \end{proof}
   
   \begin{lemma}\label{L5.4}
   For all integers $n\ge1$ we have
   \begin{equation}
   \int_{e^{f_n}}^u\frac{(\log\log t)^{n}}{\log^{n+1} t}\,dt\le 4u
   \frac{(\log\log u)^{n}}{\log^{n+1} u},\qquad (u\ge e^{f_n})
   \end{equation}
   where $f_n= 4(n+1)/3$.
   \end{lemma}
   
   \begin{proof}
   Notice that $f_n>1$. 
   For $t\ge e$ the function $\log\log t$ is positive and increasing so that
   \begin{displaymath}
   \int_{e^{f_n}}^u\frac{(\log\log t)^{n}}{\log^{n+1} t}\,dt\le (\log\log u)^{n}
   \int_{e^{f_n}}^u\frac{dt}{\log^{n+1} t}.
   \end{displaymath}
   It remains to be shown that
   \begin{displaymath}
   \int_{e^{f_n}}^u\frac{dt}{\log^{n+1} t}\le 
   \frac{4u}{\log^{n+1} u},\qquad (u\ge e^{f_n}).
   \end{displaymath}
   Replacing $u$ by $e^x$ this is equivalent to
   \begin{displaymath}
   \int_{f_n}^x\frac{e^t}{t^{n+1}}\,dt\le 
   \frac{4e^x}{x^{n+1}},\qquad  (x\ge f_n).
   \end{displaymath}   
   For the function 
   \begin{displaymath}
   G(x):=\frac{4e^x}{x^{n+1}}-\int_{f_n}^x\frac{e^t}{t^{n+1}}\,dt
   \end{displaymath}
   we have
   \begin{displaymath}
   G'(x)=\frac{e^x}{x^{n+1}}\Bigl(4-\frac{4(n+1)}{x}-1\Bigr)
   \end{displaymath}
   so that for  $x>4(n+1)/3$ we obtain  $G'(x)>0$.
   Since $G(f_n)>0$ we have $G(x)>0$ for all $x>f_n$. 
   \end{proof}

   \begin{theorem}\label{PboundP}
   The polynomials $P_n(y)$ defined in \eqref{defV} satisy the inequalities
   \begin{equation}\label{boundP}
   |P_n(y)|\le 3\cdot n!\, y^{n},
   \qquad y\ge2,\quad n\ge 1
   \end{equation}
   and $|P_0(y)|\le y$ for $y\ge2$.
   \end{theorem}
   
   \begin{proof}
   Since $P_0(y)=y-1$, the second assertion is trivial.

   Given $r>0$, for each polynomial  $P(x)=\sum_{n=0}^N a_n x^n$ we define
   \begin{displaymath}
   \Vert P\Vert=\sum_{n=0}^N |a_n| r^n.
   \end{displaymath}
   It is easy to show that 
   \begin{displaymath}
   \Vert P+Q\Vert\le \Vert P\Vert+\Vert Q\Vert,\qquad 
   \Vert PQ\Vert\le \Vert P\Vert\cdot \Vert Q\Vert
   \end{displaymath}
   and that for the derivative of a polynomial of degree $\le N$
   \begin{displaymath}
   \Vert P'\Vert=\sum_{n=0}^N n|a_n| r^{n-1}\le 
   \frac{N}{r}\sum_{n=0}^N |a_n| r^{n}=\frac{N}{r}\Vert P\Vert.
   \end{displaymath}
   For $y\ge r$ we have the inequality
   \begin{displaymath}
   |P(y)|=\Bigl|\sum_{n=0}^N a_n y^n\Bigr|\le \sum_{n=0}^N |a_n| y^n\le 
   \sum_{n=0}^N |a_n|r^n (y/r)^n\le (y/r)^N\Vert P\Vert.
   \end{displaymath}
   
   Hence, our Theorem follows if it can be shown that for $n\ge1$ we have 
   $\Vert P_n\Vert\le 3\cdot 2^n\,n!$ (for $r=2$).
   
   Define $S_n:=\Vert P_n\Vert$. By \eqref{formP} we have $S_n=-P_n(-2)$, 
   and it can be shown that $S_n\le 3\cdot 2^n\,n!$ for $0\le n\le 15$. 
   
   For $n>15$ it follows from \eqref{rec}  and the aforementioned properties 
   of $\Vert P\Vert$
   that 
   \begin{displaymath}
   S_n\le nS_{n-1}+\frac{n}{2}S_{n-1}+\frac{1}{n}\sum_{k=1}^{n-1}k
   \Bigl((k-1)S_{k-1}+S_k+\frac{k}{2}S_{k-1}\Bigr)S_{n-k-1}.
   \end{displaymath}
   It follows that  $S_n\le T_n$ where    $T_n:=S_n\le 3\cdot 3^n\,n!$ 
   for $0\le n\le 15$,  and that
   for $n>15$ 
   \begin{displaymath}
   T_n:=\frac{3n}{2}T_{n-1}+\frac{1}{n}\sum_{k=1}^{n-1}\Bigl(k T_k+\frac{k(3k-2)}{2}
   T_{k-1}\Bigr)T_{n-k-1}.
   \end{displaymath}
   Now we proceed by induction. For $n>15$ and assuming that 
   we have proved $T_k\le 3\cdot 2^k \,k!$ for $k<n$, we obtain 
   \begin{multline*}
   T_n\le   
   \frac{9n}{2}2^{n-1}(n-1)!+\\+\frac{9}{n}
   \sum_{k=1}^{n-1}\Bigl(k 2^k k!+\frac{k(3k-2)}{2}
   2^{k-1}(k-1)!\Bigr)2^{n-k-1}(n-k-1)!.
   \end{multline*}
   Hence
   \begin{multline*}
   \frac{T_n}{3\cdot 2^n\,n!}\le \frac{3}{4}+\frac{3}{n}\sum_{k=1}^{n-1}
   \Bigl(\frac{k\cdot k! (n-k-1)!}{2\cdot n!}+\frac{(3k-2) k!(n-k-1)!}{8\cdot n!}
   \Bigr)\le \\ \le 
   \frac{3}{4}+\frac{3}{8n^2}\sum_{k=1}^{n-1}\frac{7k-2}{\binom{n-1}{k}}\le 
   \frac{3}{4}+\frac{3(7n-9)}{8n^2}+\frac{3}{8n^2}
   \sum_{k=1}^{n-2}\frac{7k-2}{\binom{n-1}{k}}.
   \end{multline*}
   Therefore, by using the symmetry of the combinatorial numbers, we obtain 
   \begin{multline*}
   \frac{T_n}{3\cdot 2^n\,n!}\le \frac{3}{4}+\frac{3(7n-9)}{8n^2}+\frac{3}{16n^2}
   \sum_{k=1}^{n-2}\frac{7n-11}{\binom{n-1}{k}}\le \\ \le
   \frac{3}{4}+\frac{3(7n-9)}{8n^2}+\frac{3}{16n^2}\cdot (n-2)\frac{7n-11}{n-1}\le \\
   \le
   \frac{3}{4}+\frac{3(7n-9)}{8n^2}+\frac{3}{16n^2}\cdot (7n-11) = 
   1-\frac{n(4n-63)+87}{16n^2}<1
   \end{multline*}
   for $n>15$.    
   \end{proof}
   
   \begin{corollary}\label{little}
   We have 
   \begin{equation}\label{Pineq}
   |P_{n-1}(y)|\le n!\,y^n,\qquad n\ge1, \quad y\ge2.
   \end{equation}
   \end{corollary}
   
   \begin{proof}
   This follows easily from the above Theorem.
   \end{proof}
   
   \subsection{Main inequalities.}
   
   To simplify our formulae we introduce some notation. 
   First we set $r_n:=3\cdot n!$ so that, for $n\ge1$, we have
   $|P_n(y)|\le r_n y^n$ when $y>2$. 
   
   Let $c_n$ and $d_n$ be the constants introduced in Lemma \ref{Lsimple} and 
   Proposition \ref{usean}. Let $\alpha_n$ be equal to $\max(e, c_n,d_n)$ and let 
   $\beta_n\ge e$ be the solution of the equation
   \begin{equation}
   \frac{x}{\log x}=\alpha_n.
   \end{equation}
   (The function $\frac{t}{\log t}$ is increasing for $t\ge e$).
   
   Finally, define $x_n:=\max(\beta_n, f_n, e^2)$, where $f_n$ is defined in 
   Lemma \ref{L5.4}.
   
   \begin{proposition}\label{Pxnnumber}
   Let $x$ be a  real number such that $x\ge x_n$, and set $y:=\log x$. Then
   \begin{displaymath}
   y\ge2,\qquad x \ge c_n y,\qquad x \ge d_n y, \qquad x\ge f_n.
   \end{displaymath}
   \end{proposition}
   \begin{proof}
   Since $x\ge x_n=\max(\beta_n, f_n, e^2)$ we have $x\ge e^2$, so that
   $y = \log x\ge2$.
   
   We also have $x\ge \beta_n\ge e $. Since $\frac{t}{\log t}$ is 
   an increasing function
   for $t\ge e$    we obtain
   $\frac{x}{\log x}\ge\frac{\beta_n}{\log \beta_n}=\alpha_n=\max(e,c_n,d_n)$. 
   Therefore,
   $\frac{x}{y}\ge c_n$ and $\frac{x}{y}\ge d_n$ as required.
   \end{proof}
   
   We insert a table of the constants $x_n$.
   \bigskip
   
   \begin{tabular}{|c|l||c|l||c|l||c|l|} \hline
   $x_1$ & 7.38906 & $x_6$ & 10.81135 & $x_{11}$ & 16.00000 & $x_{20}$ &\  29.57923\\
   \hline
   $x_2$ & 7.38906 & $x_7$ & 11.70187 & $x_{12}$ & 17.33333 & $x_{30}$ &\ 47.86556\\
   \hline
   $x_3$ & 7.38906 & $x_8$ & 12.60164 & $x_{13}$ & 18.66667 & $x_{40}$ &\ 67.69154\\
   \hline
   $x_4$ & 8.29874 & $x_9$ & 13.58167 & $x_{14}$ & 20.00000 & $x_{50}$ &\ 88.57644\\
   \hline
   $x_5$ & 9.77283 & $x_{10}$ & 14.66667 & $x_{15}$ & 21.42740 & $x_{60}$ & 110.29065\\
   \hline
   \end{tabular}
   \bigskip
   
   For each natural number $N$ we set
   \begin{equation}\label{defW}
   W_N=1+\sum_{n=1}^N\frac{P_{n-1}(y)}{x^n}
   \end{equation}
   and frequently we  write $W:=W_N$  when $N$ is fixed. 
   
   \begin{proposition}\label{xeN}
   For $N\ge1$  let $W=W_N$ (as in \eqref{defW}). Then
   for $x\ge x_N$ and $y=\log x$ there exists $\theta$ with $|\theta|\le1$ 
   such that
   \begin{equation}\label{E32}
   W+xW+xW_x+W_y-x-y-\log W=\theta\cdot  r_{N+1} \frac{y^{N}}{x^N}.
   \end{equation}
   \end{proposition}
  
   \begin{proof}
   Denote by 
   $T=T(x,y)$ the value of $W+xW+xW_x+W_y-x-y-\log W$. Then we have
   \begin{multline*}
   T=(1+x)\sum_{n=0}^N \frac{P_{n-1}(y)}{x^n}
   -\sum_{n=1}^N\frac{nP_{n-1}(y)}{x^{n}}
   +\sum_{n=1}^N\frac{P'_{n-1}(y)}{x^n}-\\
   -x-y
   +\log\Bigl(1+\sum_{n=1}^N \frac{P_{n-1}(y)}{x^n}\Bigr)^{-1}.
   \end{multline*}
   From \eqref{defQ} we have  the expansion
   \begin{equation}\label{defU}
   \log\Bigl(1+\sum_{n=1}^\infty \frac{P_{n-1}(y)}{x^n}\Bigr)^{-1}
   =-\sum_{n=1}^\infty\frac{Q_n(y)}{x^n}.
   \end{equation}
   From Proposition \ref{Pxnnumber} we know that $y=\log x >2$ and $x\ge y d_N$.
   From \eqref{Pineq}, for $y>2$, we have $|P_{n-1}(y)|\le n! y^n$ so that 
   we have the 
   majorant 
   \begin{equation}\label{majorant}
   \log\Bigl(1+\sum_{n=1}^N \frac{P_{n-1}(y)}{x^n}\Bigr)^{-1}
   \ll\log\Bigl(1-\sum_{n=1}^N \frac{n!}{(x/y)^n}\Bigr)^{-1}
   \end{equation}
   (by considering this expression as a power series in $x^{-1}$, 
   and $y$ as a parameter).
   From \eqref{defU} and \eqref{majorant},
   we  obtain
   \begin{equation}\label{valuelog}
   \log\Bigl(1+\sum_{n=1}^N \frac{P_n(y)}{x^n}\Bigr)^{-1}=-
   \sum_{n=1}^N\frac{Q_n(y)}{x^n}+ S_N(x,y)
   \end{equation}
   where $S_N(x,y)$ is a power series majorized by the Taylor expansion of
   \begin{displaymath}
   \log\Bigl(1-\sum_{n=1}^N \frac{n!}{(x/y)^n}\Bigr)^{-1}-\sum_{n=1}^N
   \frac{a_n}{n}\frac{1}{(x/y)^n}
   \end{displaymath}
   (compare equation \eqref{apartial}).
   
   By applying Proposition \ref{usean} we deduce that, for $x>y d_N$, there exists
   $\theta $ with $|\theta|\le 1$ and 
   \begin{equation}
   S_N(x,y)= \theta\frac{a_{N+1}}{N+1}
   \frac{y^{N+1}}{x^{N+1}}.
   \end{equation}   
   
   If we substitute \eqref{valuelog} in the expression for $T$, then by 
   Theorem \ref{cancel},  all the terms 
   in $x^{-n}$ with $n<N$ cancel out, and the terms in $x^{-N}$ 
   add up to $-P_{N}(y)x^{-N}$. It follows that 
   \begin{equation}
   T = -\frac{P_{N}(y)}{x^N}+S_N(x,y).
   \end{equation}
   Therefore, since  $y>2$,  we  have 
   \begin{displaymath}
   |T|\le r_N\frac{  y^{N}}{x^N}+\frac{a_{N+1}}{N+1}
   \frac{y^{N+1}}{x^{N+1}}
   \end{displaymath}
   so that from  \eqref{boundan}  ,
   \begin{displaymath}
   |T|\le \frac{ y^{N}}{x^N}\Bigl(3\cdot  N!+ 2\cdot(N+1)!\frac{\log x}{x}\Bigr)\le 
   \frac{3 \cdot (N+1)! y^{N}}{x^N}=r_{N+1}\frac{y^{N}}{x^N}
   \end{displaymath}   
   where $N\ge1$ and $\frac{3}{2}+\frac{2\log x}{x}\le 3$
   for $x\ge e^2$ are applied. 
   \end{proof}
   
   \begin{proposition}\label{5.8}
   For each natural number $N$  let  $u_N=e^{x_N}$. Then 
   there exists $v_N>u_N$ such that 
   \begin{equation}
   \li(f_N(u))-u=\theta\cdot 13 (N+1)!\,
   \frac{u(\log\log u)^N}{\log^{N+1}u},\qquad (u>v_N)
   \end{equation}
   where $|\theta|\le1$ and 
   \begin{equation}\label{EE:5.20}
   f_N(e^x):=xe^xW_N(x,\log x)=
   xe^x\Bigl(1+\sum_{n=1}^N\frac{P_{n-1}(\log x)}{x^n}\Bigr).
   \end{equation}   
   \end{proposition}
   
   \begin{proof}
   To simplify the notation, the abbreviation $W(x,y)=W_N(x,y)$ is used. 
   Differentiating \eqref{EE:5.20} we 
   obtain
   \begin{multline*}
   \frac{d}{dx}\bigl(\li(f_N(e^x))-e^x\bigr)=\\=
   \frac{1}{\log(f_N(e^x))}\Bigl\{e^x W+xe^x W+xe^x\Bigl(W_x+\frac{1}{x}W_y\Bigr)
   \Bigr\}-e^x.
   \end{multline*}
   Assume that $x\ge x_N$,  so that $x\ge d_N\log x$ and $x\ge e^2$. 
   We may apply \eqref{E32} to obtain 
   \begin{multline*}
   \frac{d}{dx}\bigl(\li(f_N(e^x))-e^x\bigr)=
   \frac{e^x}{\log(f_N(e^x))}\Bigl\{ W+x W+xW_x+W_y\Bigr\}-e^x=\\
   =\frac{e^x}{\log(f_N(e^x))}\Bigl\{x+\log x+ \log W+\theta  r_{N+1}\frac{
   \log^N x}{x^N}\Bigr\}-e^x.
   \end{multline*}
   This may be simplified to
   \begin{displaymath}
   \frac{d}{dx}\bigl(\li(f_N(e^x))-e^x\bigr)=\frac{e^x}{\log(f_N(e^x))}
   \cdot \theta\frac{r_{N+1}\,\log^N x}{x^N}.
   \end{displaymath}
   Since $x>x_N$ we  have $x\ge y c_N$, so that by Lemma \ref{Lsimple} 
   \begin{displaymath}
   \Bigl|x\Bigl(1+\sum_{n=1}^N\frac{P_{n-1}(\log x)}{x^n}\Bigr)\Bigr|\ge x\Bigl(1-
   \sum_{n=1}^N\frac{n!}{(x/y)^n}\Bigr)\ge1
   \end{displaymath}
   that is
   $xW_N(x,\log x)\ge1$, so that
   $\log(f_N(e^x))\ge x$. 
   Hence, for $x\ge x_N$ (with another $\theta$), we have
   \begin{displaymath}
   \frac{d}{dx}\bigl(\li(f_N(e^x))-e^x\bigr)=
   \theta\frac{r_{N+1}\,\log^N x}{x^{N+1}}e^x.
   \end{displaymath}
   Defining $H_N(u):=\li(f_N(u))-u$ the above equation is
   equivalent to 
   \begin{displaymath}
   H_N'(e^x)=\theta\frac{r_{N+1}\,\log^N x}{x^{N+1}},\qquad  (x\ge x_N)
   \end{displaymath}
   and, since $u_N:=e^{x_N}$,
   \begin{displaymath}
   H_N'(u)=\theta\frac{r_{N+1}\,(\log\log u)^N}{\log^{N+1}u},\qquad  (u\ge u_N).
   \end{displaymath}
   
   Lemma \ref{L5.4} can be applied since $x_N\ge f_N$, so that $u\ge u_N\ge e^{f_N}$.
   Hence, integrating over the interval $(u_N,u)$ we get
   \begin{displaymath}
   H_N(u)=H_N(u_N)+\theta
   \frac{4r_{N+1}\,u(\log\log u)^N}{\log^{N+1}u},\qquad (u\ge u_N).
   \end{displaymath}
   
   The function $(N+1)!\frac{u}{\log^{N+1}u}\cdot (\log\log u)^N$ is increasing 
   (as product of two positive increasing functions) 
   for $u>e^{f_n}$, so that there exists 
   $v_N>u_N$ for which this function is greater than $H_N(u_N)$, so that
   \begin{displaymath}
   H_N(u)=\theta
   \frac{13\cdot(N+1)!\,u(\log\log u)^N}{\log^{N+1}u},\qquad (u\ge v_N).
   \end{displaymath}
   \end{proof}
   
   \begin{remark}\label{remark54} For the values of $n$ appearing 
   in our tables, the equality $u_n = v_n$ holds, since, in these cases,
   \begin{displaymath}
   H_n(u_n)\le \frac{(n+1)!\,u_n(\log\log u_n)^n}{\log^{n+1}u_n}.
   \end{displaymath}
   \end{remark}

   \begin{lemma}\label{boundlogfN}
   For any natural number $N$, and  $u>e^{x_N}$  we have $\log f_N(u)<2\log u$.
   \end{lemma}
   
   \begin{proof}
   First observe that the hypothesis $u>e^{x_N}$ implies (with $u=e^x$) that 
   $x>x_N$, so that $\log x>2$ and $x>c_N\log x$. (Proposition \ref{Pxnnumber}). 
   
   The inequality $\log f_N(u)<2\log u$ is equivalent to $f_N(u)<u^2$, and together
   with $u=e^x$ it is equivalent to 
   \begin{displaymath}
   xe^x\Bigl(1+\sum_{n=1}^N \frac{P_{n-1}(\log x)}{x^n}\Bigr)<e^{2x}.
   \end{displaymath}
   From Corollary \ref{little}, since $x\ge e^2$ and $x\ge c_N y$,  and by
   Remark \ref{remark54},  
   \begin{displaymath}
   x\Bigl(1+\sum_{n=1}^N \frac{P_{n-1}(\log x)}{x^n}\Bigr)\le 
   x\Bigl(1+\sum_{n=1}^N \frac{n!}{(x/y)^n}\Bigr)<
   x\Bigl(1-\sum_{n=1}^N \frac{n!}{(x/y)^n}\Bigr)^{-1}.
   \end{displaymath}
   Hence our inequality follows from 
   \begin{displaymath}
   x<\frac{ye^x}{x}\cdot\frac{x}{y}\Bigl(1-\sum_{n=1}^N \frac{n!}{(x/y)^n}\Bigr),\qquad
   y=\log x.
   \end{displaymath}
   Since we assume that $x\ge y c_N$, the second factor  is greater than $1$, 
   so that
   \begin{displaymath}
   \frac{ye^x}{x}\cdot\frac{x}{y}\Bigl(1-\sum_{n=1}^N \frac{n!}{(x/y)^n}\Bigr)>
   \frac{ye^x}{x}
   \end{displaymath}
   Finally, it is easy to prove that $e^x\log x>x^2$ for $x>e^2$.   
   \end{proof}
   
   The asymptotic expansion with bounds can now be proved.
   
   \begin{theorem}\label{TMain}
   For each integer $N\ge1$
   \begin{equation}
   \ali(u) = f_N(u) + 26\theta(N+1)!\,u\Bigl(\frac{\log\log u}{\log u}\Bigr)^N,\qquad 
   (u\ge v_N),
   \end{equation}
   where $v_N$ is the number defined  in Proposition \ref{5.8} .
   \end{theorem}
   
   \begin{proof}
   Since $\li(\ali(u))=u$,  Proposition \ref{5.8} yields, for $u>v_N$,
   \begin{displaymath}
   \li(f_N(u))-\li(\ali(u))=\int_{\ali(u)}^{f_N(u)}\frac{dt}{\log t}=
   13\theta(N+1)!\,u\frac{(\log\log u)^N}{\log^{N+1} u}.
   \end{displaymath}
   Since $v_N\ge u_N=e^{x_N}$, $u\ge v_N$ implies $\log u\ge 2$, hence $u\ge 2$.
   
   From Lemma \ref{comparation},  $\log\ali(u)\le 2\log u$, for 
   $u>2$. Analogously, Lemma \ref{boundlogfN} implies that 
   $\log f_N(u)\le 2\log u$, for $u>e^{x_N}$.  
   Therefore,  for $u>v_N$, we  have 
   \begin{displaymath}
   \frac{|\ali(u)-f_N(u)|}{2\log u}\le \Bigl|\int_{\ali(u)}^{f_N(u)}
   \frac{dt}{\log t}\Bigr|.
   \end{displaymath}
   It follows that there exists $\theta'$ with $|\theta'|\le 1$ such that 
   \begin{displaymath}
   \ali(u)-f_N(u)=\theta'(2\log u)\int_{\ali(u)}^{f_N(u)}\frac{dt}{\log t}
   \end{displaymath}
   and the result follows easily.
   \end{proof}
   
   The actual error appears to be much smaller than that  given in Theorem
   \ref{TMain}.  However, as usual with asymptotic expansions, having a 
   true bound allows  realistic bounds to be given of the remainder
   for specific values of $N$.
   
   The true error after $N$ terms of an asymptotic expansion, while the
   terms are decreasing in magnitude, is often of the size of the first
   omitted term.  In our case, the magnitude of the term 
   $P_{N}(\log x) x^{-N-1}$ depends on the polynomial $P_N(\log x)$.
   
   Numerically, it appears that  for $n\ge3$:
   \begin{equation}\label{conj}
   |P_n(y)|\le \Bigl(\frac{n}{e\log n}\Bigr)^n y^n,\qquad (y>2\log n)
   \end{equation}
   although we have not been able to prove this.   
   
   From Theorem \ref{TMain},  more realistic bounds can be obtained for the 
   first values of $N$. This is done in the following Theorem.
   
   \begin{theorem}\label{concrete}
   For $2\le N\le 11$, we have
   \begin{equation}\label{realistic}
   \frac{\ali(e^x)}{xe^x}=1+\sum_{n=1}^{N}\frac{P_{n-1}(\log x)}{x^n}+
   \theta\cdot20\cdot
   \Bigl(\frac{N}{e\log N}\Bigr)^N\cdot
   \frac{\log^{N}x}{x^{N+1}},\qquad (x>z_N)
   \end{equation}
   where 
   \begin{gather*}
   z_2=1.50,\quad z_3=2.34,\quad z_4=3.32,\quad z_5=4.33,\quad z_6 = 5.36,\\
   z_7=6.39,\quad z_8=7.43,\quad z_9=8.46,\quad z_{10}=9.50,\quad z_{11}=10.53.
   \end{gather*}
   \end{theorem}
   
   \begin{proof}
   By taking $N=10$ in Theorem \ref{TMain}, we have, for $u=e^x>e^{x_{10}}$,
   (recall also  Remark \ref{remark514})
   \begin{displaymath}
   \frac{\ali(e^x)}{xe^x}=1+\sum_{n=1}^{10}\frac{P_{n-1}(\log x)}{x^n}+
   \theta R \frac{\log^{10}x}{x^{11}},\qquad (x>x_{10})
   \end{displaymath}
   with $R=26\cdot11!=1\,037\,836\,800$. 
   
   We compute the maximum\footnote{$M_2=1$, $M_3=1/2$, $M_4=1/3$, 
   $M_5=0.250636$, $M_6=0.526887$, $M_7 = 1.300565$, $M_8 = 3.719653$, $M_9=12.070813$,
   $M_{10}=43.788782$. This last maximum would be much smaller if the maximum 
   were taken
   from a point slighthly greater than $x_{10}$.}
   $M_n$ of $|P_{n-1}(\log x)/\log^{n-1}x|$ 
   for $x>x_{10}$, so that for any $2\le N\le 10$, we  have
   \begin{multline*}
   \frac{\ali(e^x)}{xe^x}=1+\sum_{n=1}^{N}\frac{P_{n-1}(\log x)}{x^n}+\\
   +\frac{\log^{N}x}{x^{N+1}}\Bigl(\sum_{n=N+1}^{10}\frac{P_{n-1}(\log x)}{\log ^{n-1}x}
   \frac{\log^{n-N-1}x}{x^{n-N-1}}+\theta
   R \frac{\log^{10-N}x}{x^{10-N}}\Bigr)
   \end{multline*}
   so that  
   \begin{multline*}
   \frac{\ali(e^x)}{xe^x}=1+\sum_{n=1}^{N}\frac{P_{n-1}(\log x)}{x^n}+\\
   +\theta
   \frac{\log^{N}x}{x^{N+1}}\Bigl(\sum_{n=N+1}^{10}
   \frac{M_n\log^{n-N-1}x}{x^{n-N-1}}+R\frac{\log^{10-N}x}{x^{10-N}}\Bigr).
   \end{multline*}
   We determine a value $z'_N>x_{10}$ such that,  for $x=z'_N$,
   \begin{displaymath}
   \Bigl(\sum_{n=N+1}^{10}
   \frac{M_n\log^{n-N-1}x}{x^{n-N-1}}+R\frac{\log^{10-N}x}{x^{10-N}}\Bigr)<
   20\Bigl(\frac{N}{e\log N}\Bigr)^N.
   \end{displaymath}
   Since this is a decreasing function of $x$, we obtain for $x>z'_N$
   \begin{equation}\label{desire}
   \frac{\ali(e^x)}{xe^x}=1+\sum_{n=1}^{N}\frac{P_{n-1}(\log x)}{x^n}+\theta\cdot
   20\Bigl(\frac{N}{e\log N}\Bigr)^N\frac{\log^{N}x}{x^{N+1}}.
   \end{equation}
   We consider the function
   \begin{displaymath}
   \Bigl(\frac{\ali(e^x)}{xe^x}-1-\sum_{n=1}^{N}\frac{P_{n-1}(\log x)}{x^n}
   \Bigr)\frac{x^{N+1}}{\log^N x}
   \end{displaymath}
   on
   the  interval $(1.3,z'_{N})$, to determine the least value of $z_{N}$ 
   for which \eqref{desire} is true. 
   
   In this way we find: $z'_2=32$ and then $z_2=1.5$; $z'_3=49.5$ and then
   $z_3=2.3395$; $z'_4=82$ and then $z_4=3.3114$;  $z'_5=155$ and then
   $z_5=4.3237$.  

   If we take $N=20$ in Theorem \ref{TMain}, we obtain 
   $z'_6=113$, $z'_7=143$, $z'_8=187$, $z'_9=251$
   $z'_{10}=353$, $z'_{11}=528$ from which
   $z_6 = 5.3514$, $z_7=6.3851$, $z_8=7.4208$,
   $z_9=8.4566$, $z_{10}=9.4914$ 
   and $z_{11}=10.5251$ are obtained.   \end{proof}

   \begin{remark}
   We have proved \eqref{realistic} only for $2\le N\le 11$, although something
   similar appears to be true for the general case.   If \eqref{realistic} were true for 
   all $n$, then for a large $u=e^x$ we could take $N\approx x$ terms in the expansion
   and in this way the  error would be $\approx\frac{20}{x e^x}$, 
   so that $\ali(u)$ could be computed with 
   an error less than $\approx20$. 
   
   In fact, for several values of $u$, the terms of the expansion have been computed
   up to the point 
   where these terms start to increase. Always  the computation is terminated
   when $N\approx x$ and the error
   appears to be bounded.   ( For example, with $u=10^{100}$, we compute 230 terms
   of the expansion, which coincides with $\log10^{100}\approx 230.259$. 
   The approximate value 
   obtained
   for $\ali(u)$ has an absolute error 
   equal to $40.94738$, which can be compared with the fact that 
   $\ali(10^{100})$ has 103 digits ).

   \end{remark}

   \section{Applications to $p_n$.}\label{S4}
   
   \subsection{Asymptotic expansion of $p_n$.}
   Inequalities for the $n$-th prime number can be found in 
    \cite{R}, \cite{Ro},  \cite{MR},  \cite{D}. 
   In fact, from $\pi(x)=\li(x)+
   \Orden(r(x))$ we may obtain $p_n=\ali(n)+\Orden(r(n\log n)\log n)$, if 
   $r(x)/x$ is sufficiently small.  For example, in \cite{MR}, it is noticed
   that from a result of Massias \cite{M}, it follows that
   \begin{equation}
   p_n=\ali(n)+\Orden(n e^{-c\sqrt{\log n}})
   \end{equation}
   so that the asymptotic expansion of $\ali(n)$ is also an
   asymptotic expansion for $p_n$, that is,
   \begin{equation}
   p_n=n\log n\Bigl(1+\sum_{k=1}^N\frac{P_{k-1}(\log\log n)}{\log ^k n}\Bigr)+
   \Orden\Bigl(n \Bigl(\frac{\log\log n}{\log n}\Bigr)^N\Bigr).
   \end{equation}
  
   By assuming the Riemann hypothesis,  Schoenfeld \cite{S} has proved
   \begin{equation}\label{es67Sch}   
   |\pi(x)-\li(x)|<\frac{1}{8\pi}\sqrt{x}\log x,\qquad (x>2657).
   \end{equation}
   This result will be used to obtain (under RH) some precise bounds for $p_n$. 
   
   \begin{lemma}\label{Ldos}
   We have    
   $\sqrt{x}(\log x)^{\frac52}<\ali(x)$ for $x>94$. 
   \end{lemma}
   
   \begin{proof}
   The inequality is equivalent to 
   \begin{displaymath}
   \li(\sqrt{x}(\log x)^{\frac52})<x.
   \end{displaymath}
   Differentiating the function $f(x):=x-\li(\sqrt{x}(\log x)^{\frac52})$
   we get
   \begin{displaymath}
   f'(x)=1-\frac{1}{\log(\sqrt{x}(\log x)^{\frac52})}
   \Bigl(\frac{\log^{5/2}x}{2\sqrt{x}}
   +\frac{5\log^{3/2}x}{2\sqrt{x}}\Bigr).
   \end{displaymath}
   Hence this derivative is positive if and only if
   \begin{displaymath}
   \log^{3/2}x\Bigl(\frac{\log x}{2}+\frac52\Bigr)<\sqrt{x}
   \Bigl(\frac{\log x}{2}+\frac52\log\log x\Bigr).
   \end{displaymath}
   For $x>94$ we have $(\log x)^{\frac32}<\sqrt{x}$, so that
   $f'(x)>0$ for $x>94$. 
   
   Finally, one may verify that $f(94)>0$.  Hence $f(x)>0$ for $x>94$.
   \end{proof}   
   
   \begin{theorem}\label{Tdistance}
   The Riemann hypothesis is equivalent to the assertion
   \begin{equation}\label{eqTdistance}
   |p_n-\ali(n)|<\frac{1}{\pi}\sqrt{n}\log^{\frac52} n\qquad  
   \text{for all } n\ge11.
   \end{equation}
   \end{theorem}
   
   \begin{proof}
   First we assume the Riemann Hypothesis and prove \eqref{eqTdistance}.
   Let $r(x):=\frac{1}{8\pi}\sqrt{x}\log x$,  $f(x):=\li(x)-r(x)$, and 
   $g(x):=\li(x)+r(x)$. For $x>1$ we have $f(x)<\li(x)<g(x)$, where the
   three functions are strictly increasing. From \eqref{es67Sch} for 
   $x>2657$, we also have  
   $f(x)<\pi(x)<g(x)$. 
   
   The inverse functions satisfy $g^{-1}(y)<\ali(y)<f^{-1}(y)$, and if 
   $y=n>\pi(2657)=384$ is a natural number, then 
   $g^{-1}(n)<p_n<f^{-1}(n)$. It follows that the distance from $\ali(n)$
   to $p_n$ is bounded by 
   \begin{displaymath}
   |p_n-\ali(n)|\le\max\bigl(f^{-1}(n)-\ali(n), \ali(n)-g^{-1}(n)\bigr).
   \end{displaymath}
   
   Hence, we have to  bound $f^{-1}(y)-\ali(y)$ and $\ali(y)-g^{-1}(y)$.
   
   We consider $y$ as a parameter and set $\alpha=\ali(y)$, so that
   $\li(\alpha)=\li(\ali(y))=y$. 
   
   Consider
   the function $u(\xi):=f(\xi)-\li(\alpha)=f(\xi)-y$, which is 
   strictly increasing and
   satisfies 
   \begin{displaymath}
   u(\xi)=\li(\xi)-r(\xi)-\li(\alpha)=\int_\alpha^\xi\frac{dt}{\log t}-r(\xi).
   \end{displaymath}
   Therefore, $u(\alpha)=-r(\alpha)<0$ and 
   \begin{displaymath}
   u(f^{-1}(y))=f(f^{-1}(y))-\li(\alpha)=y-\li(\ali(y))=0.
   \end{displaymath}
   If a point $b$ is found where $u(b)>0$, then 
   $\alpha<f^{-1}(y)<b$, so that $b-\alpha>f^{-1}(y)-\alpha$ and  one of the
   required bounds is obtained. 
   
   Therefore, we try $b = \alpha+c\sqrt{y}(\log y)^{\frac52}$ with $c<1$.
   We have
   \begin{multline*}
   u(\alpha+c\sqrt{y}(\log y)^{\frac52})=
   \int_\alpha^{\alpha+c\sqrt{y}(\log y)^{\frac52}}
   \frac{dt}{\log t}-r(\alpha+c\sqrt{y}(\log y)^{\frac52})>\\
   >\frac{c\sqrt{y}(\log y)^{\frac52}}{\log(\alpha+c\sqrt{y}(\log y)^{\frac52})}
   -r(\alpha+c\sqrt{y}(\log y)^{\frac52}).
   \end{multline*}
   From Lemma \ref{Ldos} for $y>94$, we  have $\sqrt{y}(\log y)^{\frac52}<
   \ali(y)=\alpha$, so that
   \begin{multline*}
   u(\alpha+c\sqrt{y}(\log y)^{\frac52})>
   \frac{c\sqrt{y}(\log y)^{\frac52}}{\log(2\alpha)}
   -r(2\alpha) \\
   = \frac{c\sqrt{y}(\log y)^{\frac52}-\frac{1}{8\pi}\sqrt{2\alpha}
   \log^2(2\alpha)}{\log(2\alpha)}.
   \end{multline*}
   We want to show that this expression is positive.
   For $y>94$, we have
   $\alpha=\ali(y)<2y\log y$ (by Lemma \ref{comparation}), 
   so that $\alpha<2y\log y<4y\log y<y^2$ (for $y>94$), which yields (with $c=1/\pi$)
   \begin{multline*}
   c\sqrt{y}(\log y)^{\frac52}-\frac{1}{8\pi}\sqrt{2\alpha}
   \log^2(2\alpha)>\\
   >c\sqrt{y}(\log y)^{\frac52}-\frac{1}{8\pi}\sqrt{4y\log y}
   \log^2(4y\log y)>\\ >
   c\sqrt{y}(\log y)^{\frac52}-\frac{1}{8\pi}\sqrt{4y\log y}
   \log^2(y^2))=0.
   \end{multline*}

   Hence, we have proved that $f^{-1}(y)-\alpha<\frac{1}{\pi}\sqrt{y}(\log y)^{5/2}$
   for $y>94$.
   \medskip
   
   To bound $\alpha-g^{-1}(y)$, we consider the function
   $v(\xi):=g(\xi)-\li(\alpha)=g(\xi)-y$. Then
   \begin{displaymath}
   v(\xi)=r(\xi)-\int_\xi^\alpha\frac{dt}{\log t}
   \end{displaymath}
   and $v(\alpha)=r(\alpha)>0$, $v(g^{-1}(y))=g(g^{-1}(y))-y =0$. 
   If  a value $b$ is found such that $v(b)<0$, it will follow that 
   $\alpha-g^{-1}(y)<\alpha-b$.  
   
   Choose $b=\alpha-c\sqrt{y}(\log y)^{5/2}$ with $c=\frac{1}{\pi}$.  We claim that 
   $v(b)<0$. We have
   \begin{displaymath}
   v(b)=r(b)-\int_b^\alpha\frac{dt}{\log t}<r(b)-\frac{\alpha-b}{\log \alpha}
   \end{displaymath}
   and our claim will follow from 
   $r(b)\log\alpha-c\sqrt{y}(\log y)^{5/2}<0$. 
   Finally, since $b<\alpha=\ali(y)<2y\log y$, and by Lemma \ref{comparation},
   \begin{displaymath}
   r(b)\log\alpha<\frac{1}{8\pi}\sqrt{\ali(y)}(\log\ali(y))^2
   <\frac{1}{8\pi}\sqrt{2y\log y}(2\log y)^2
   \end{displaymath}
   which proves our claim. 
   
   Hence, (by assuming RH),  we have proved that $|p_n-\ali(n)|<
   \frac{1}{\pi}\sqrt{n}(\log n)^{\frac52}$ for 
   $n>384>94$. By verifying all  $1\le n\le 385$, we find that the inequality
   holds except for $n<11$.
   
   The reverse implication is simple. From $|p_n-\ali(n)|=
   \Orden(n^{\frac12+\varepsilon})$ for any $\varepsilon>0$,  
   we may derive that $\pi(x)=\li(x)+
   \Orden(x^{\frac12+\varepsilon})$. 
   It is well known that this is equivalent to the Riemann Hypothesis. 
   \end{proof}
   
   \begin{remark}
   The inequality \eqref{eqTdistance} is only proved by assuming the
   Riemann Hypothesis, but is stronger than those contained in 
   \cite{R}, \cite{Ro},  \cite{MR},  \cite{D}. Inequality  \eqref{eqTdistance}
   gives approximately half of the digits of $p_n$. If  our conjecture 
   that \eqref{realistic} is true for all $N$ is also assumed, then
   the asymptotic expansion  gives  about
   half of the digits of $p_n$. 
   \end{remark}
   
   \subsection{Inequalities for the $n$-th prime.}
   Let 
   \begin{displaymath}
   s_N(n)=n\log n\Bigl(1+\sum_{k=1}^N\frac{P_{k-1}(\log\log n)}{\log ^k n}\Bigr)
   \end{displaymath}
   where $s_0=n\log n$. 
   
   Cipolla noted that for $k\ge1$, $P_{k}(y)=(-1)^{k+1} \frac{y^k}{k}+\cdots$, 
   and $P_0(y)=y-1$.  Hence, except for the first term, eventually the sign
   of the $k$-th term $P_{k-1}(\log\log n)\log^{-k}n$ becomes $(-1)^k$.  The asymptotic
   expansion implies  that there exist $r_N$ such that 
   \begin{equation}\label{ineqCipolla}
   \begin{split}
   p_n>s_0(n),\quad  n>r_0,\qquad p_n>s_1(n),\quad  n>r_1,\\
   p_n<s_{2N}(n),\quad n>r_{2N},\qquad  p_n>s_{2N+1}(n),\quad n>r_{2N+1}.
   \end{split}
   \end{equation}
   In fact, $r_0=2$ is the main result in \cite{R}, $r_1=2$ is proved in \cite{D}
   and $r_2=688\,383$ is proved in \cite{D2}.
   The value of $r_N$ for $N\ge3$ has not been determined. See Theorem  
   \ref{T:r3}
   for an estimation of  $r_3$ by assuming RH. 
      
   The above reasoning may give the impression that the terms of the asymptotic
   expansion of $\ali(u)$ are alternating in sign, starting from the second term.
   However this is not true.  For example, computing the first 230 terms for 
   $\ali(10^{100})$,
   we found only three positive terms $P_0/x$, $P_1/x^2$,  and $P_3/x^4$.
   In fact, the sign of the $k$-th term is that of $P_{k-1}(\log\log n)$. 
   Thus we are interested in  the
   sign  of these  polynomials.
   
      The polynomials $P_N(y)$, for $1\le N\le 23$ of odd index, have one and 
   only one  real root, which is positive. Starting from $P_1(y)$ which 
   vanishes at $y=2$, these roots are 
   \begin{gather*}
   2,\quad 4.23415,\quad 5.83131,\quad 7.43591,\quad 9.07979,
   \quad 10.6881,\quad 12.2538,\\
    13.7876,\quad 15.2977,
   \quad 16.79, \quad 18.2683, \quad 19.7353.
   \end{gather*}
   The polynomials $P_2$, $P_4$ and $P_6$ have no real roots,  and all
   $P_8$, \dots , $P_{22}$ 
   have two positive real roots. These pairs of roots are:
   \begin{gather*}
   (6.4306, 8.2185),\quad (7.16158, 9.88528), \quad
   (7.90293, 11.4752),\\ (8.63359, 13.0241),\quad
   (9.3507, 14.5452),\quad (10.055, 16.0458),\\
   (10.7478, 17.5307),\quad (11.4307, 19.003).
   \end{gather*}
   
   For example, $P_9(\log\log n)$ is positive only for $n> \exp(e^{9.07\dots})$,
   which is a very big number.
   
   The even terms  at first sight appear \emph{negative}. However 
   $P_{10}(\log\log n)$, for
   example,  is negative except in
   the interval $\exp(e^{7.16\dots})<n<\exp(e^{9.88\dots})$.

   In a certain sense, the inequalities \eqref{ineqCipolla} are the wrong
   inequalities.  
   These inequalities would hold  
   only for very large values of $r_N$, especially
   when we want a lower bound of $p_n>s_{2N+1}$ (except for the three
   known cases). We estimate $r_3$.

   \begin{theorem}\label{T:r3}
   Let $r_3$ be the smallest number such that 
   \begin{multline}
   p_n>s_3:=n\log n+n(\log\log n-1)+n\frac{\log\log n-2}{\log n}-\\
   -
   n\frac{(\log\log n)^2-6\log\log n+11}{2\log^2n}, \qquad  n\ge r_3.
   \end{multline}
   Then, if the Riemann Hypothesis is assumed, 
   \begin{displaymath}
   39\times10^{29}<r_3 \le 39.58\times10^{29}.
   \end{displaymath}
   \end{theorem}
   
   \begin{proof}
   By Theorem \ref{Tdistance}, there exists $\theta_1$ with $|\theta_1|\le 1$
   such that 
   \begin{displaymath}
   p_n=\ali(n)+\theta_1\frac{\sqrt{n}}{\pi}(\log n)^{\frac52},\qquad n\ge11.
   \end{displaymath}
   
   By Theorem \ref{concrete},  with $5\le N\le 10$ and  setting 
   $n=e^x$, $x=\log n$, and 
   $y=\log\log n$, we have for $n>e^{z_N}$ 
   \begin{multline*}
   \ali(n)=x e^x\Bigl(1+\frac{y-1}{x}+\frac{y-2}{x^2}
   -\frac{y^2-6y+11}{2 x^3}+\\
   +\frac{P_3(y)}{x^4}+\sum_{k=5}^N\frac{P_{k-1}(y)}{x^k}
   +\theta_2 c_N\frac{y^N}{x^{N+1}}\Bigr).
   \end{multline*}
   
   The inequality of the Theorem is obtained if 
   \begin{displaymath}
   x e^x\Bigl(
   \frac{P_3(y)}{x^4}+\sum_{k=5}^N\frac{P_{k-1}(y)}{x^k}
   - c_N\frac{y^N}{x^{N+1}}\Bigr)-\frac{e^{x/2}}{\pi}x^{\frac52}>0.
   \end{displaymath}
   This is equivalent to 
   \begin{displaymath}
   P_3(y)+\sum_{k=5}^NP_{k-1}(y)e^{-(k-4)y}>
   c_Ny^Ne^{-(N-3)y}+\frac{1}{\pi}e^{\frac{11}{2}y}e^{-\frac12e^{y}}.
   \end{displaymath}
   Since $P_3(y)\to+\infty$ and all the other terms tend to $0$ 
   as $y\to+\infty$,
   it is clear that the inequality is true for $y>y_0$. 
   With $N=10$, we find $y_0=4.254946453\dots$ The inequality $p_n < s_n$ is
   true for $n\ge  3.95702224148845656\times 10^{30}$.
   This proves that $r_3\le 39.58\times10^{29}$. 
   
   In order to show that $r_3>39\times10^{29}$, we directly show that,
   for $n=39\times10^{29}$, the opposite inequality 
   $p_n<s_3$ is obtained. 
   
   We compute 
   \begin{displaymath}
   s_3=    2.87527\,18639\,02974\,79681\,42399\,35057\,89294\,02005\,
   87915 \times 10^{32}.
   \end{displaymath}
   Now we can compute $\ali(n)$,
   for which we already have obtained a good approximation through 
   the asymptotic expansion, and then apply the Newton method
   \begin{displaymath}
   \ali(n)=2.87527\,18639\,02495\,21516\,14800\,14732\,45414\,39731\times 10^{32}.
   \end{displaymath}
   Therefore, from Theorem \ref{Tdistance}, we obtain $p_n<\ali(n)+\frac{1}{\pi}
   \sqrt{n}\log^{\frac52}n$, so that
   \begin{displaymath}
   p_n< 2.87527\,18639\,02756\,97808\,39055\,05640\,30082\,86370\,11482
   \times 10^{32}
   \end{displaymath}
   and we can  conclude that  $p_n<s_3$.
   \end{proof}

\end{document}